\documentclass[a4paper,12pt]{amsart}
\usepackage[utf8]{inputenc}
\usepackage[T1]{fontenc}
\usepackage[UKenglish]{babel}
\usepackage[margin=18mm]{geometry}

\usepackage{color}%cyan,red;magenta,green;yellow,blue

\usepackage{graphicx}

\usepackage{amsmath,amssymb,amsfonts,amsthm}
\usepackage{mathrsfs,eucal,dsfont}
\usepackage{verbatim,enumitem}

\usepackage{hyperref,url}

\renewcommand{\leq}{\leqslant}

\newcommand{\R}{\mathds R}

\newcommand{\I}{\mathds 1}

\def\d{{\rm d}}
\def\<{\langle}
\def\>{\rangle}

\def\R{\mathbb R}    
 \def\kk{\kappa} 
  \def\vv{\varepsilon} 
\def\<{\langle} \def\>{\rangle}  
  \def\nn{\nabla}  
\def\d{\text{\rm{d}}}   
   
 \def\beq{\begin{equation}}  
 
\def\e{\text{\rm{e}}}  \def\OO{\Omega}  
  
 \def\P{\mathbb P}

\def\E{\mathbb E}

\def\to{\rightarrow}
\def\8{\infty}\def\3{\triangle}
\def\1{\lesssim}

\renewcommand{\bar}{\overline}
\renewcommand{\hat}{\widehat}

\newtheorem{theorem}{Theorem}[section]
\newtheorem{lemma}[theorem]{Lemma}
\newtheorem{proposition}[theorem]{Proposition}
\newtheorem{corollary}[theorem]{Corollary}

\theoremstyle{definition}

\newtheorem{example}[theorem]{Example}
\newtheorem{remark}[theorem]{Remark}

\numberwithin{equation}{section}
\begin{document}
\allowdisplaybreaks

\title[Ergodicity of modified Euler schemes] {Geometric  ergodicity of modified Euler schemes for SDEs with super-linearity}

\author{
Jianhai Bao\qquad Mateusz B. Majka \qquad
Jian Wang}
\date{}
\thanks{\emph{J.\ Bao:} Center for Applied Mathematics, Tianjin University, 300072  Tianjin, P.R. China. \url{jianhaibao@tju.edu.cn}}

\thanks{\emph{M.B. \ Majka:} School of Mathematical and Computer Sciences, Heriot-Watt University, Edinburgh, EH14 4AS, UK.\url{m.majka@hw.ac.uk}}

\thanks{\emph{J.\ Wang:}
School  of Mathematics and Statistics \& Key Laboratory of Analytical Mathematics and Applications (Ministry of Education) \& Fujian Provincial Key Laboratory
of Statistics and Artificial Intelligence, Fujian Normal University, 350007 Fuzhou, P.R. China. \url{jianwang@fjnu.edu.cn}}

\maketitle

\begin{abstract}
As a well-known fact, the classical Euler scheme works merely for SDEs with coefficients of linear growth. In this paper, we
study a general framework of modified Euler schemes,
 which is applicable to SDEs with super-linear drifts and encompasses
 numerical methods such as the tamed Euler scheme and the truncated Euler scheme.
 On the one hand, by exploiting an approach based on the refined basic coupling, we show that
 all Euler recursions within our proposed framework are
 geometrically ergodic under a mixed probability distance (i.e., the total variation distance plus the $L^1$-Wasserstein distance) and the weighted total variation distance.  On the other hand, by utilizing
 the coupling by reflection, we
  demonstrate that the tamed Euler scheme is geometrically ergodic under the $L^1$-Wasserstein distance. In addition,
as an important application,
we provide
   a quantitative $L^1$-Wasserstein  error bound  between the exact invariant probability measure
   of an SDE with super-linearity,
   and the invariant probability measure of the tamed Euler scheme which is its numerical counterpart.

\medskip


\smallskip

\end{abstract}
\section{Introduction and main result}
In the past few decades,
there have been significant advances in the area of numerical approximations for SDEs.
As far as  the convergence analysis in a finite-time
horizon is concerned,
the theory has progressed well beyond the classical Euler-Maruyama (EM) scheme, and
 we refer to \cite{dS,HMS} for the backward  EM scheme,  \cite{HJK,KNRS,Sabanis} regarding the tamed EM scheme, \cite{LMS,LMTW,Mao} concerning the truncated EM method, and \cite{FG,RS} with regard to the adaptive EM scheme, to name just a few.
For results on the behaviour of EM schemes in the infinite-time horizon, the reader may consult \cite{FG,LMS,TKLN}.
 In particular, recently there has been a series of works, where the authors utilized the probabilistic coupling method in order to investigate the long-time behaviour of EM schemes under the assumption of  ``dissipativity at infinity'' concerned with  the drifts of the associated SDEs.
 Based on quantitative contraction rates for Markov chains on a general state space, the exponential contractivity under the Wasserstein distance for EM schemes associated with non-degenerate SDEs driven by Brownian motions and L\'{e}vy noises was explored in \cite{EM} and \cite{HMW} by drawing on the reflection coupling and the refined basic coupling, respectively. Via the synchronous coupling, along with the construction of
 a corresponding
 metric,  the $L^2$-Wasserstein contraction for Euler schemes was established in \cite{LMM} for SDEs with a sufficiently high diffusivity. Through a mixture of  the synchronous coupling and  the reflection coupling, the   $L^1$-Wasserstein contraction for the EM scheme corresponding to  kinetic Langevin samplers with  non-convex potentials was treated in \cite{SW}. Leveraging on discrete sticky couplings, \cite{DEEGM} provided
 bounds in Wasserstein and weighted
 total variation distances between  distributions of EM schemes and  corresponding perturbation versions
 for SDEs with Lipschitz drifts, which are dissipative at large distances. Additionally, concerned with the tamed unadjusted Langevin algorithm,
 \cite{BDMS} demonstrated the $L^2$-Wasserstein non-asymptotic error bound between the exact invariant probability measure and its numerical version, for overdamped  Langevin SDEs, albeit only with strongly convex potentials.

With the exception of \cite{BDMS}, which treated tamed Euler schemes, all those results were applicable only to the case of SDEs with global Lipschitz drifts.
 Inspired by the aforementioned literature, in this work, we shall go further and focus on SDEs with super-linear drifts. For such
 SDEs, we intend to propose a novel
 framework for studying EM schemes, which generalizes the classical EM scheme and encompasses
   the tamed EM scheme  and the truncated EM scheme, as typical candidates. Most importantly, by constructing  appropriate couplings, we shall consider geometric ergodicity  under different probability distances for the proposed EM
 iteration.
 Compared to \cite{DEEGM,EM, HMW, LMS}, our approach allows for the study of EM schemes for SDEs with super-linear drifts. In comparison with \cite{BDMS}, our approach is much more general and provides bounds in the Wasserstein distances without requiring strong convexity of the potential.

Among different types of modified EM schemes, recently there has been a lot of interest especially in tamed EM schemes, due to their applications in computational statistics and machine learning. In addition to \cite{BDMS}, we refer to  e.g. \cite{LLRS, LNSZ, NNZ, LS} for a series of papers on the Tamed Unadjusted Langevin Algorithm (TULA) and its variants. These papers are concerned with error bounds between the target probability measure and the distribution of the tamed EM scheme after a certain number of steps. Typically, the target measures of interest are of the form $\pi \propto \e^{-U}$ for some potential $U: \R^d \to \R$ and, if $U$ has growth stronger than quadratic, the classical ULA cannot be applied directly and taming is required. The majority of such error bounds in the literature rely on contraction properties of the associated SDEs (which are usually obtained via couplings or functional inequalities). However, this does not always provide precise control on the constants for the error of the discrete process;  see e.g. \cite[Remark 2.12]{LNSZ} or \cite[Remark 2.12]{NNZ}. In the present paper, we apply couplings directly at the level of the discrete-time scheme, which has the potential of providing sharper bounds, with a better control of the constants. See \cite{MMS} for related results in the classical case of $U$ with quadratic growth (i.e., Lipschitz $\nabla U$), where bounds for  ULA and Stochastic Gradient Langevin Dynamics (SGLD) algorithms were obtained by applying couplings directly at the level of the discrete process, in contrast to the majority of the literature. We believe that the results obtained in the present paper will provide an opportunity for obtaining sharper bounds for algorithms such as TULA. However, this will require a detailed analysis and is beyond the scope of the present work, whose focus is instead on providing a general framework for analysing ergodicity of modified EM schemes.

Consider an SDE on $\R^d:$
\begin{equation}\label{E1}
\d X_t=b(X_t)\,\d t+ \d W_t,
\end{equation}
where $b:\R^d\to\R^d$ is measurable,
and $(W_t)_{t\ge0}$ is a  $d$-dimensional Brownian motion on the filtered probability space
$(\OO,\mathscr F,(\mathscr F_t)_{t\ge0},\P)$. Throughout the paper, we shall assume that the SDE \eqref{E1} is strongly well-posed under suitable conditions (e.g., $b$ is locally Lipschitz and satisfies a Lyapunov condition; see, for instance,  \cite[Theorem 3.1.1]{PR}).

 In most
  scenarios, SDEs under consideration are unsolvable explicitly unless they  possess certain special structures. Based on the  point of view above,
plenty of numerical schemes  are proposed  to approximate numerically SDEs  and simulate with the aid of  computers.
Concerning  the SDE \eqref{E1}, we put forward the following explicit discretization scheme: for the iteration step number   $n\ge0$ and the step size $\delta>0$,
\begin{equation}\label{E2}
X_{(n+1)\delta}^{\delta}=\pi^{(\delta)}(X_{n\delta}^\delta)+b^{(\delta)}\big(\pi^{(\delta)}(X_{n\delta}^\delta)\big)\delta +\triangle W_{n\delta},
\end{equation}
where $\triangle W_{n\delta}:=W_{(n+1)\delta}-  W_{n\delta}$ means the  increment (with the length $\delta$) of $(W_t)_{t\ge0}$;
the mapping $\pi^{(\delta)}:\R^d\to\R^d$ is contractive, that is,
\begin{align}\label{E3}
\big|\pi^{(\delta)}(x)-\pi^{(\delta)}(y)\big|\le |x-y|,\quad x,y\in\R^d;
\end{align}
$b^{(\delta)}:\R^d\to\R^d$ is a modified version of $b$ so that, for each fixed   $x\in\R^d$, $|b^{(\delta)}(x)-b(x)|\rightarrow0$ as $\delta\rightarrow0$. Since Brownian motions have the self-similar property, there exist i.i.d. $d$-dimensional standard Gaussian random variables $(\xi_n)_{n\ge1}$, supported on $(\OO,\mathscr F, \P)$,
such that, for each integer $n\ge0$, $\triangle W_{n\delta}$  and $\delta^{{1}/{2}}\xi_{n+1}$ are distributed identically.
 Thus, in \eqref{E2}, by replacing $\triangle W_{n\delta}$ with $\delta^{{1}/{2}}\xi_{n+1}$, we obtain the following modified  EM  scheme: for  any $\delta>0$ and integer $n\ge0$,
\begin{equation}\label{EW}
X_{(n+1)\delta}^{\delta}=\pi^{(\delta)}(X_{n\delta}^\delta)+b^{(\delta)}\big(\pi^{(\delta)}(X_{n\delta}^\delta)\big)\delta +\delta^{{1}/{2}} \xi_{n+1}.
\end{equation}
 At first glance,  the stochastic algorithm  \eqref{EW} is    a
 bit unusual. Nevertheless, it
 provides a flexible framework that incorporates several well-known stochastic algorithms. Indeed, as we will discuss in detail in Section \ref{sec2}, our framework covers:
 \begin{itemize}
 	\item The classical EM scheme (Subsection \ref{subsec:classicalEM}), by taking $\pi^{(\delta)}(x) = x$ and $b^{(\delta)}(x) = b(x)$;
 	\item The tamed EM scheme (Subsection \ref{subsec:tamedEM}), by taking $\pi^{(\delta)}(x) = x$ and $b^{(\delta)}(x) = \frac{b(x)}{1+\delta^{\theta}\varphi(|x|)}$, where $\theta \in (0,1/2)$ and $\varphi : [0,\infty) \to [0,\infty)$ is a strictly increasing and  continuous function with $\varphi(0) = 0$ and a polynomial growth;
 	\item The truncated EM scheme (also known in the literature as the projected EM scheme; Subsection \ref{sec2.3}), by taking $\pi^{(\delta)}(x) = \frac{1}{|x|}\left( |x| \wedge \varphi^{-1}(\delta^{-\theta}) \right) x \I_{\{ |x|>0 \}}$, where $\theta \in (0,1/2)$ and $\varphi : [0,\infty) \to [0,\infty)$ is a  strictly increasing and continuous function, whereas $b^{(\delta)}(x) = b(x)$.
 \end{itemize}

\subsection{Geometric ergodicity under a mixed   probability distance}
In this subsection,  we assume that
\begin{enumerate}
\item[(${\bf H}_1$)]there exist constants $R, C_R,K_R\ge1$,   $\theta\in(0,1/2)$, and $\delta_0\in(0,1]$ such that for
all
 $\delta\in(0,\delta_0]$, and all  $x,y\in\R^d,$
\begin{align}\label{EW4}
\big|b^{(\delta)}(\pi^{(\delta)}(x))-b^{(\delta)}(\pi^{(\delta)}(y))\big|\le \big(C_R \I_{\{x\in B_R\}\cap\{y\in B_R\}}+K_R\delta^{-\theta}\I_{\{x\in B_R^c\}\cup\{y\in B_R^c\}}\big)|x-y|;
\end{align}
\item[(${\bf H}_2$)] for the threshold
$R\ge1$
 above,
there exists a  constant   $ K^*_R >0$   such that for all $\delta\in(0,\delta_0]$, and $x,y\in\R^d $  with $x\in B_R^c$  or $y\in B_R^c$,
\begin{align}\label{EW5}
\big\< \pi^{(\delta)}(x)-\pi^{(\delta)}(y), b^{(\delta)}(\pi^{(\delta)}(x))-b^{(\delta)}(\pi^{(\delta)}(y))\big\>\le -K^*_R \big| \pi^{(\delta)}(x)-\pi^{(\delta)}(y)\big|^2,
\end{align}
where $B_R $ denotes the closed ball with radius $R$ and center ${\bf0}\in\R^d,$ and $B_R^c$ means its complement.
\end{enumerate}

Before we proceed,  we make some comments on Assumptions $({\bf H}_1)$ and $({\bf H}_2)$.
\begin{remark}
Assumption $({\bf H}_1)$ demonstrates that the composition of $b^{(\delta)}$ and $\pi^{(\delta)}$ satisfies
a global Lipschitz condition, in which, outside of the closed ball $B_R$, the Lipschitz constant depends on the step size.
Inequality \eqref{EW5}
  reveals that  $b^{(\delta)}$ is $K_R^*$-dissipative outside of $B_R$, where, most importantly,  $K_R^*$ is
  independent of
  the step size  $\delta.$
A more intuitive counterpart of \eqref{EW5} would be the following one:
 for any $x,y\in\R^d $  with $x\in B_R^c$  or $y\in B_R^c$,
\begin{align}\label{EWR}
\big\<x-y, b^{(\delta)}(x)-b^{(\delta)}(y)\big\>\le -K^*_R |x-y|^2.
\end{align}
Nevertheless, for the proof of Theorem \ref{thm} below, the technical condition \eqref{EW5}  is essential.
Since the precise expression of $\pi^{(\delta)}$ is unknown, $x\in B_R^c$ does not
automatically imply $\pi^{(\delta)}(x)\in B_R^c$.
 Hence, \eqref{EW5} cannot be
derived from  \eqref{EWR} so, in Assumption $({\bf H}_1)$, \eqref{EW5} (in lieu of \eqref{EWR}) is in force.
However, in some specific cases we can impose \eqref{EWR} as long as $\pi^{\delta}$ possesses an explicit form;
see Section \ref{sec2.3} for relevant details. Last but not least,
 in order to demonstrate the robustness of Assumptions $({\bf H}_1)$ and $({\bf H}_2)$, we discuss in detail in Section \ref{sec2} how to verify them for the classical EM scheme, the tamed EM scheme, and the truncated EM scheme.
\end{remark}

In the following paragraph, we introduce some
notations.
Let $\nu(\d z)$ be the law of the $d$-dimensional standard Gaussian random variables $(\xi_n)_{n\ge1}$, i.e.,
 $$\nu(\d z)=(2\pi)^{-{d}/{2}}\e^{-{|z|^2}/{2}}\,\d z,   $$
and,
 for $x\in\R^d$, let $\nu_x(\d z)=(\nu\wedge (\delta_x*\nu))(\d z),$ which stands for
the common part of $\nu(\d z)$ and $(\delta_x*\nu)(\d z)$ (the convolution between $\nu(\cdot)$ and Dirac's delta measure $\delta_x(\cdot)$).
 A direct calculation shows that for all $x\in\R^d,$
 \begin{align*}
 \nu_x(\d z)=(2\pi)^{-{d}/{2}}\Big(\e^{-{|z|^2}/{2}}\wedge \e^{- {|z-x|^2}/{2}}\Big)\,\d z.
 \end{align*}
For any $r\ge0,$ define
$$J(r)=\inf_{|u|\le r}\nu_u(\R^d).$$
It is easy to see that $J(r)>0$ for any finite number  $r \ge0$
 by making use of properties of the Gaussian distribution $\nu$.
Additionally, we define the distance function
\begin{align}\label{EW1}
 \rho_1(x,y)=  \I_{\{x\neq y\}}+ |x-y|,\quad x,y\in\R^d,
\end{align}
and  set
\begin{equation}\label{EE6*}
\delta_1:=\delta_0\wedge\big(K_R^*/K_R^2\big)^{\frac{1}{1-2\theta}}\wedge  (2K_R^*)^{-1}\wedge (4RC_R )^{-2},
\end{equation}
where the parameters $R,C_R,K_R\ge1$,
 $\delta_0>0$,
as well as $K_R^*>0$
 are given in $({\bf H}_1)$ and $({\bf H}_2)$, respectively.
At last,  on some occasions,
we shall write $(X_{n\delta}^{\delta,x})_{n\ge0}$ instead of $(X_{n\delta}^{\delta})_{n\ge0}$   once   the initial value   $X_0^\delta=x$ is to be emphasized.

For any lower semi-continuous function $\rho: \mathbb{R}^d\times \mathbb{R}^d \to \mathbb{R}_+$, and for all probability measures $\mu$ and $\nu$ on $\R^d$, we define the $L^1$-Wasserstein (pseudo-)distance induced by $\rho$ as
\begin{equation}\label{eq:defWasserstein}
 \mathbb W_{\rho}(\mu,\nu):= \inf_{\pi \in \Pi(\mu,\nu)} \int_{\mathbb{R}^d\times \mathbb{R}^d} \rho(x,y) \pi(\d  x,\d  y)  \,,
\end{equation}
where $\Pi(\mu,\nu)$ represents the set of all couplings between the probability measures $\mu$ and $\nu$.

The main result in this subsection is stated as follows, which shows that
$(X_{n\delta}^{\delta })_{n\ge0}$ is geometrically ergodic under the Wasserstein distance induced by the cost function $\rho_1$ defined in \eqref{EW1}.

\begin{theorem}\label{thm}
Assume that \eqref{E3},
 $({\bf H}_1)$ and $({\bf H}_2)$ hold. Then,  there exists a constant $C_1^* >0$ such that  for any $\delta\in(0, \delta_1]$, $x,y\in\R^d,$ and integer $n\ge0,$
\begin{align}\label{W}
\mathbb W_{\rho_1}\big(\mathscr L_{X_{n\delta}^{\delta,x}},\mathscr L_{ X_{n\delta}^{\delta,y}}\big)\le C_1^* \e^{-\lambda_1 n\delta}\rho_1(x,y),
\end{align}
 where
 $\mathbb W_{\rho_1}$ stands for the Wasserstein distance induced by the cost function $\rho_1$ defined in \eqref{EW1}, and $\lambda_1:=\lambda_{11}\wedge\lambda_{12}\wedge\lambda_{13}\in(0,1) $  with
 \begin{align}\label{EW*}
\lambda_{11}:=\frac{aJ(1) }{4(a+2 )},\quad \lambda_{12}: =\frac{4RC_Rc_*\e^{-2Rc_*} }{ a+2 }  \quad \mbox{ and }
\quad \lambda_{13}: =\frac{ R c_*K_R^*\e^{-2Rc_* }  }{a+2}
\end{align}
with
\begin{align}\label{EP6}
c_*:=\frac{32RC_R(1+C_R)^2}{ J(1)}\quad \mbox{ and } \quad a:=\frac{8c_*}{J(1)}+2C_R +
K_R^*.
\end{align}
\end{theorem}

 Evidently,
 Theorem \ref{thm} implies the following corollary.
\begin{corollary}\label{cor}
Assume  that \eqref{E3},
 $({\bf H}_1)$ and $({\bf H}_2)$ hold. Then, there exists a constant $C_1^*>0$ such that
 for any $\delta\in(0, \delta_1]$, $x,y\in\R^d,$ and $n\ge0,$
\begin{align}\label{Cor1}
\big\|\mathscr L_{X_{n\delta}^{\delta,x}}-\mathscr L_{ X_{n\delta}^{\delta,y}}\big\|_{\rm{var}}+\mathbb W_1\big(\mathscr L_{X_{n\delta}^{\delta,x}},\mathscr L_{ X_{n\delta}^{\delta,y}}\big)\le  C_1^*\e^{-\lambda_1 n\delta}
\rho_1(x,y),
\end{align}
where $\lambda_1>0$ is the same as that  in Theorem $\ref{thm}$, $\|\cdot\|_{\rm{var}}$ means the total variation distance, and $\mathbb W_1$ denotes the $L^1$-Wasserstein distance.
\end{corollary}

Note that Corollary \ref{cor} shows that $(X_{n\delta}^\delta)_{n\ge0}$ is exponentially decaying  under the total variation distance and the standard $L^1$-Wasserstein distance. However, it is not exponentially contractive in either distance, since the term $\rho_1(x,y)$ on the right hand side of \eqref{Cor1} cannot be controlled from above by either only $c \I_{\{ x \neq y\}}$  or only  $c|x-y|$ for some constant $c > 0$ (one needs to use the sum of both terms). In particular, we refer to  the introductory  section of   \cite{LMM} for a recent discussion on the contractivity results for classical Euler-Maruyama schemes that are currently available in the literature, as well as on the applications of such results, where it is important to have contraction rather than just upper bounds. In the next subsections, we turn our attention to showing exponential contractivity (rather than only exponential decay).

\subsection{Geometric ergodicity under the additive Wasserstein  distance}
 In this subsection, we turn to investigate the
 exponential contractivity
 of $(X_{n\delta}^\delta)_{n\ge0}$ under the weighted total variation distance, which is equivalent to the additive Wasserstein distance.
 More precisely, we work with the distance function:
 \begin{align}\label{rho2}
 \rho_2(x,y):=(2+|x|^2+|y|^2)\I_{\{x\neq y\}},\quad x,y\in\R^d.
 \end{align}

To begin, we introduce some additional notations.
We assume
 \begin{align}\label{binfty}
b^{(\8)}_0:=\sup_{\delta\in(0,\delta_0]}\big|b^{(\delta)}({\bf0})\big|<\8.
\end{align}
 Define the subset $\mathcal D$ of $\R^d\times\R^d$ by
\begin{align}\label{EP4}
\mathcal D=\big\{(x,y)\in\R^d\times\R^d:2+ |x|^2+|y|^2<8c_0/K_R^*\big\},
\end{align}
where
\begin{equation}\label{EP3}
\begin{split}
c_0:=d+&\big((1+2C_R+2C_R^2+K_R^*/2)R^2+3\big(b^{(\8)}_0\big)^2\big)\\
&\vee\big(1+1/(2^{1-2\theta}-1)+2/K_R^*\big)\big(b^{(\8)}_0\big)^2\big).
\end{split}
\end{equation}
It is easy to see that the diameter of $\mathcal D$ is finite, i.e.,
\begin{align*}
r_{\mathcal D}:=\sup_{ (x,y)\in \mathcal D }|x-y|<\8.
\end{align*}
Furthermore, in the following, we shall  stipulate
\begin{align}\label{T0}
c_\star =\frac{8}{J(1)}C_R(1+C_R)^2(1+r_{\mathcal D})
\end{align}
and fix respectively the positive parameters $\vv_\star$ and $a$ as follows:
\begin{align}\label{T3}
\vv_\star:=\frac{J(1)c_\star^2  \e^{-c_\star(1+r_{\mathcal D }) }}{32c_0(1+C_R)^2}\quad   \mbox{ and } \quad    a:=\frac{4}{J(1)}( c_\star+2c_0\vv_\star)+2\Big(\frac{2}{K_R^*}\vee1\Big)c_0c_\star\vv_\star+\frac{1}{4}K_R^*c_\star.
\end{align}
Finally, we  set
\begin{align*}
\delta_2:=\delta_0\wedge  \Big(\frac{1}{2}\big(K_R^*/K_R^2\big)^{\frac{1}{1-2\theta}}\Big)\wedge  \frac{a J(1)}{2K_R^*(a+1)}\wedge (4RC_R )^{-2}.
\end{align*}

The main result in this subsection is the following.

\begin{theorem}\label{add}
Assume that \eqref{E3} with $\pi({\bf0})={\bf0}$,
 $({\bf H}_1)$,
$({\bf H}_2)$ and \eqref{binfty} hold.
 Then,   there exists a constant $C_2^*>0$ such that for all $\delta\in(0, \delta_2]$, $x,y\in\R^d,$ and integer $n\ge0,$
\begin{align}\label{ER6}
\mathbb W_{\rho_2}\big(\mathscr L_{X_{n\delta}^{\delta,x}},\mathscr L_{ X_{n\delta}^{\delta,y}}\big)\le C_2^*\e^{-\lambda_2 n\delta}\rho_2(x,y)
\end{align}
 where $\mathbb W_{\rho_2}$
 denotes the Wasserstein distance induced by the metric function
 $\rho_2$ defined in \eqref{rho2},
and    $\lambda_2:=\lambda_{21}\wedge\lambda_{22}\wedge\lambda_{23}\in(0,1)$ with
\begin{align*}
 \lambda_{21}:=\frac{1}{2}K_R^*,\quad \lambda_{22}:=\frac{2c_0c_\star\vv_\star}{a+1}  \quad \mbox{ and } \quad \lambda_{23}:=\frac{K_R^* \vv_\star }{4(a+1)}.
\end{align*}
\end{theorem}

Let $\mathscr P_2(\R^d)$ be the set of probability measures on $\R^d$ with finite second moment.
According to \cite[Lemma 2.1]{HM}, the weighted total  variation distance:
\begin{align*}
\|\mu-\nu\|_{2,\rm{var}}:=\sup_{|f|\le 1+|\cdot|^2}\bigg|\int_{\R^d}f(x)\mu(\d x)-\int_{\R^d}f(x)\nu(\d x)\bigg|,\quad \mu,\nu\in\mathscr P_2(\R^d)
\end{align*}
is identical  to the Wasserstein distance $\mathbb W_{ \rho_2}$.  Thus, as a direct byproduct of Theorem \ref{add}, $(X_{n\delta}^\delta)_{n\ge0}$ is geometrically  ergodic under the weighted total variation distance $\|\cdot\|_{2,\hbox{var}}$.
\begin{corollary}
Assume that Assumptions of Theorem $\ref{add}$ hold. Then,    for all $\delta\in(0, \delta_2]$, $x,y\in\R^d$ and $n\ge0,$
\begin{align*}
\big\|\mathscr L_{X_{n\delta}^{\delta,x}}-\mathscr L_{ X_{n\delta}^{\delta,y}}\big\|_{2,\rm{var}}\le C_2^* \e^{-\lambda_2 n\delta}\rho_2(x,y),
\end{align*}
where $C_2^*>0$ and $\lambda_2\in(0,1)$ are  the same as those in Theorem $\ref{add}$, respectively.
\end{corollary}

\subsection{ Geometric ergodicity  under the $L^1$-Wasserstein distance}
In this subsection we consider the scheme  corresponding to the choice of $\pi^{(\delta)}(x) = x$ in \eqref{EW}. More precisely,
for any $\delta>0$ and integer $n\ge0,$
\begin{equation}\label{EW-}
X_{(n+1)\delta}^{\delta}=X_{n\delta}^\delta+b^{(\delta)}\big(X_{n\delta}^\delta\big)\delta +\delta^{1/2} \xi_{n+1},
\end{equation}
where $(\xi_n)_{n\ge1}$ are $d$-dimensional standard Gaussian random variables.

To investigate the  geometric  contractivity of $( X_{n\delta}^\delta )_{n\ge0}$, we
rewrite Assumptions (${\bf H}_1$) and (${\bf H}_2$)  as follows (since $\pi^{(\delta)}(x)=x$ for all $x\in \R^d$):
\begin{enumerate}
\item[(${\bf H}_1'$)]there exist constants $R, C_R,K_R\ge1$,   $\theta\in(0,1/2)$, and $\delta_0\in(0,1]$ such that for all $\delta\in(0,\delta_0]$, and   $x,y\in\R^d,$
\begin{align}\label{EW4-}
\big|b^{(\delta)}(x)-b^{(\delta)}(y)\big|\le \big(C_R \I_{\{x\in B_R\}\cap\{y\in B_R\}}+K_R\delta^{-\theta}\I_{\{x\in B_R^c\}\cup\{y\in B_R^c\}}\big)|x-y|;
\end{align}
\item[(${\bf H}_2'$)]  there is a  constant   $ K^*_R >0$   such that for all $\delta\in(0,\delta_0]$, and $x,y\in\R^d $  with $x\in B_R^c$  or $y\in B_R^c$,
\begin{align}\label{EW5-}
\< x-y, b^{(\delta)}(x)-b^{(\delta)}(y)\>\le -K^*_R | x-y|^2.
\end{align}
\end{enumerate}

To proceed,   some notations need to be introduced. For parameters $R,C_R,K_R\ge1$ and $K_R^*$ given respectively in (${\bf H}_1'$) and (${\bf H}_2'$), we
 set
\begin{align}\label{ET6}
c^*:=1+\frac{16RC_R}{r^*_0 r_0} +\frac{\ln (C_RK_R^*)}{2R}
\end{align}
where
\begin{align}\label{EW9}
r_0:=    1/2+2R(C_R\vee K_R),~  r^*_0:=\frac{1}{(2\pi)^{{1}/{2}}}\big(\e^{-\frac{9}{8}r_0^2}-\e^{-2r_0^2}\big),~ \gamma:=4\big((1/2+3R(C_R\vee K_R)\big).
\end{align}
Furthermore, we define
\begin{align}\label{EE0}
\delta_3=\delta_0\wedge\big( 2R(C_R\wedge K_R)\big)^{\frac{2}{2\theta-1}}\wedge \big(K_R^*/K_R^2\big)^{\frac{1}{1-2\theta }}\wedge\bigg(\frac{\ln 2}{c^*\gamma}\bigg)^2,
\end{align}
with the constants $\delta_0$, $R$, $C_R$, $K_R$, $K_R^*$ and $\theta$ defined in  $({\bf H}_1')$ and $({\bf H}_2')$.

The following theorem shows that $(\mathscr L_{X_{n\delta}^\delta})_{n\ge0}$ corresponding to \eqref{EW-} is exponentially contractive
 under the $L^1$-Wasserstein distance.

\begin{theorem}\label{W1}
Under  $({\bf H}_1')$ and $({\bf H}_2')$,  there exists  a constant $ C_3^* >0$ such that for all $\delta\in(0,\delta_3]$, integer $n\ge0,$
and $x,y\in\R^d,$
\begin{align}\label{W1_contractivity}
\mathbb W_1\big(\mathscr L_{X_{n\delta}^{\delta,x}},\mathscr L_{ X_{n\delta}^{\delta,y}}\big)\le C_3^*\e^{-\lambda_3 n\delta}|x-y|,
\end{align}
where $(X_{n\delta}^{\delta,x})_{n\ge0}$  is determined by \eqref{EW-} and
$\lambda_3:=\lambda_{31}\wedge\lambda_{32}\in(0,1)$ with
\begin{align}\label{EE7}
\lambda_{31}:= C_R\e^{-2Rc^* }\quad\mbox{ and }\quad \lambda_{32}:=\frac{1}{2} K_R^*\e^{-c^*(2R+\gamma)}.
\end{align}
\end{theorem}

Before the end of this subsection, we give a remark  regarding Theorem \ref{W1}.
\begin{remark}
In Theorem \ref{W1},
we treat the geometric ergodicity of $(X_{n\delta}^{\delta})_{n\ge0}$ given by \eqref{EW-} (i.e., \eqref{EW} with $\pi^{(\delta)}(x)=x$).
Obviously, the framework \eqref{EW} is much more general than \eqref{EW-}.
It is a natural question to ask why in Theorem \ref{W1} we work only with the scheme \eqref{EW-} instead of \eqref{EW}.
The main technical reason is that in the proof of Theorem \ref{W1},
for $x,y\in\R^d,$ the distance between $\pi^{\delta}(x)+b^{(\delta)}(\pi^{\delta}(x))-(\pi^{\delta}(y)+b^{(\delta)}(\pi^{\delta}(y)))$
and $x-y$ should approach
zero as the step size goes to zero (see \eqref{ET1} for more details).
 However, this may not necessarily hold if the only thing we know about $\pi^{(\delta)}$ is that it is contractive.
 In particular, for the $\pi^{(\delta)}$ which corresponds to the truncated EM scheme defined in \eqref{ER3} below,
 this property is not satisfied. For this reason, in this subsection we consider only the scheme \eqref{EW-} rather than  \eqref{EW}.
\end{remark}

\subsection{Error bound between   invariant probability measures}\label{section1.3}
In this subsection, as an application of Theorem \ref{W1}, we provide a quantitative convergence rate between the exact invariant probability measure and the numerical counterpart associated with the underlying tamed EM scheme \eqref{EW-}.

Below, let $\varphi:[0,\infty)\to [0,\infty)$ be a continuous and strictly increasing function with $\varphi(0)=0$ and with a polynomial growth. We assume that
\begin{enumerate}
\item[(${\bf A}_1$)]there exist constants $L_1,L_2>0$
such that for all $x,y\in\R^d,$
\begin{align}\label{EE9}
|b(x)-b(y)|\le L_1\big(1+\varphi(|x|)+\varphi(|y|)\big)|x-y|,
\end{align}
and
\begin{align}\label{EE10}
\big|b(x)\varphi(|y|)-b(y)\varphi(|x|)\big|\le L_2\big(1+\varphi(|x|)+\varphi(|y|)+\varphi(|x|)\varphi(|y|)\big)|x-y|;
\end{align}
\item[(${\bf A}_2$)]there exist constants $L_3,L_4,L_5>0$ and $R\ge0$ such that  for all $x,y\in\R^d$ with $x\in B_R^c$  or $y\in B_R^c$,
\begin{align}\label{EE11}
\<x-y,b(x)-b(y)\>\le -L_3\big(1+\varphi(|x|)+\varphi(|y|)\big)|x-y|^2,
\end{align}
and
\begin{align}\label{EE12}
\big\<x-y, b(x)\varphi(|y|)-b(y)\varphi(|x|)\big\>\le \big(L_4\big(1+\varphi(|x|)+\varphi(|y|)\big)-L_5 \varphi(|x|)\varphi(|y|)\big) |x-y|^2.
\end{align}
\end{enumerate}

 From (${\bf A}_1$),  the drift $b$ is allowed to be of polynomial growth with the
 order $r\varphi(r)$.
To handle the difficulty arising from the highly nonlinear property of $b$,
we need to modify  the drift $b$  so that its corresponding  variant, written as $b^{(\delta)}$,  is of linear growth at most. By regarding $b^{(\delta)}$ as a new drift, we can construct a time discretization scheme corresponding to the SDE \eqref{E1}.
For this, we first introduce the modified (or tamed)   drift $b^{(\delta)}$ defined as below:
\begin{align}\label{E7}
b^{(\delta)}(x)=\frac{b(x)}{1+\delta^{\theta}\varphi(|x|)},\qquad x\in\R^d,
\end{align}
where $\theta\in(0,1/2)$.
Thus, the tamed EM algorithm related  to \eqref{E1} can be constructed  as follows: for $\delta>0$ and integer $n\ge0$,
\begin{align}\label{ER1}
X_{(n+1)\delta}^{\delta}= X_{n\delta}^\delta +b^{(\delta)}(X_{n\delta}^\delta)\delta+\delta^{1/2}\xi_{n+1},
\end{align}
where $(\xi_n)_{n\ge1}$ are $d$-dimensional standard Gaussian random variables.
For the parameters involved in $({\bf A}_1)$ and $({\bf A}_2)$, we
set
 \begin{align}\label{EE16}
 \delta_4:= \delta_3\wedge \bigg(\frac{L_3}{2(1+L_1^2)}\bigg)^{\frac{1}{1-\theta}}\wedge(2/L_3),
 \end{align}
where  $ \delta_3$ is defined as in \eqref{EE0} with
$$\delta_0=(L_3/(2L_4))^{\frac{1}{\theta}}\wedge1,~
C_R=(L_1\vee L_2 )(1+\varphi(R))^2,~ K_R=L_1\wedge L_2~\mbox{ and }  ~K_R^*=(L_3/2)\wedge L_5.$$

The following theorem provides an error bound between the exact invariant probability measure and the numerical version associated with \eqref{E1} and the algorithm \eqref{ER1}, respectively.

\begin{theorem}\label{IPM}
Assume  $({\bf A}_1)$ and $({\bf A}_2)$, where  $\varphi:[0,\8)\to[0,\8)$
 such that for some constants $c_*,l^*>0$,
\begin{align}\label{ED}
\varphi(r)\le c_*\big(1+r^{l^*}\big),\quad r\ge0.
\end{align}
Then,
there exists a constant $C>0$ such that for all $\delta\in(0, \delta_4],$
\begin{align}\label{EE13}
\mathbb W_1\big(\pi,\pi^{\delta}\big)\le C\delta^\theta,
\end{align}
where $\theta\in(0,1/2),$
 $\pi\in\mathscr P_1(\R^d)$ and $\pi^{\delta}\in\mathscr P_1(\R^d)$ stand respectively for the unique invariant probability measures of $(X_t)_{t\ge0}$ and $(X_{n\delta}^{\delta})_{n\ge0}$, which is determined by \eqref{ER1}.
\end{theorem}

\begin{remark}
	Results similar to Theorem \ref{IPM} have been recently proved in \cite{DE}. However, our framework covers more general tamed Euler schemes, our assumptions are weaker, and in our Theorem \ref{W1} we provide a proof of the contractivity result for the tamed Euler scheme (which in \cite{DE} is stipulated as an assumption); see \cite[Theorem 5 and Theorem 7]{DE} for more details. For a recent discussion on the problem of quantifying long-time approximation errors for a large class of stochastic processes, the reader may also consult \cite{SS}.
\end{remark}

The content of this paper is arranged as follows. In Section \ref{sec2}, we show that the classical Euler scheme, the tamed scheme, and the truncated Euler scheme can be reformulated  respectively  as three representatives of the modified Euler scheme \eqref{EW}. Moreover, as far as three schemes mentioned previously are concerned,
 the contractive property of the mapping $\pi^{\delta}$, and technical Assumptions  (${\bf H}_1$) and (${\bf H}_2$) are examined in Section \ref{sec2}.  By construction an appropriate refined basic coupling, we complete the proof of Theorem \ref{thm} in Section \ref{sec3}.   Based on an examination of the Lyapunov condition in the semigroup form concerning the modified Euler scheme, the proof of Theorem \ref{add} is handled  in Section \ref{section_add} via the refined basic coupling as well.
Through  an application of the coupling by reflection, the proof of Theorem \ref{W1}  is  completed in Section \ref{sec4}. The last section is devoted to the proof of Theorem \ref{IPM} on the basis of the asymptotic coupling by reflection.

\section{Verification of \eqref{E3}, (${\bf H}_1$) and (${\bf H}_2$)}\label{sec2}
This section is devoted to showing that the classical EM scheme, the tamed EM scheme, and the truncated scheme, as three typical representatives,
are special cases of the stochastic algorithm \eqref{EW}. In addition, we demonstrate that the contractive property \eqref{E3}, and
Assumptions (${\bf H}_1$) and (${\bf H}_2$) can be  fulfilled
by the three aforementioned schemes.

\subsection{The classical EM scheme}\label{subsec:classicalEM}
It is well-known that the classical  EM method works merely for SDEs with coefficients of linear growth. So,
in this subsection, we shall assume that the drift $b$ is globally Lipschitz but dissipative in the long distance. In detail,
  we shall suppose that
\begin{enumerate}
\item[(${\bf B}_1$)] there exists an $L>0$ such that for all $x,y\in\R^d,$
 $$|b(x)-b(y)|\le L|x-y|;$$

\item[(${\bf B}_2$)]there exist constants $R>0$ and $K_R>0$ such that for all $x,y\in\R^d$ with $x\in B_R^c$  or $y\in B_R^c$,
\begin{align*}
\<x-y,b(x)-b(y)\>\le -K_R|x-y|^2.
\end{align*}
\end{enumerate}

The classical EM scheme associated with   \eqref{E1} is given as follows: for $\delta>0 $ and integer $n\ge0,$
\begin{align}\label{ER}
X_{(n+1)\delta}^{\delta}= X_{n\delta}^\delta +b(X_{n\delta}^\delta)\delta+\delta^{1/2} \xi_{n+1},
\end{align}
where $(\xi_n)_{n\ge1}$ are $d$-dimensional standard Gaussian random variables.
Concerning the EM scheme above, we examine \eqref{E3}, (${\bf H}_1$) as well as  (${\bf H}_2$), separately.
Once we take $\pi^{(\delta)}(x)=x$, which satisfies \eqref{E3} trivially, and subsequently choose $b^\delta(x)=b(x)$,  \eqref{E2} goes back to  \eqref{ER}. Due to  $\pi^{(\delta)}(x)=x$ and $b^{(\delta)}(x)=b(x)$,  $({\bf H}_2)$ with $K_R^*=K_R$ follows   from  $({\bf B}_2)$  right away.
In addition, by virtue of (${\bf B}_1$) and (${\bf B}_2$), it is easy to see that for all $x,y\in\R^d,$
\begin{align*}
\<x-y,b(x)-b(y)\>\le \big(L \I_{\{x\in B_R\}\cap\{y\in B_R\}}+K_R\I_{\{x\in B_R^c\}\cup\{y\in B_R^c\}}\big)|x-y|.
\end{align*}
Whence, (${\bf H}_1$) is verifiable  for any $\delta\in(0,1].$

\subsection{The tamed EM scheme in Subsection \ref{section1.3} satisfying Assumptions {\bf(A1)} and {\bf(A2)}}\label{subsec:tamedEM}
Below, let $\delta_0=(L_3/(2L_4))^{\frac{1}{\theta}}\wedge1$ for $\theta\in(0,1/2)$ and set $\delta\in(0,\delta_0].$
As  for the tamed EM scheme \eqref{ER1}, we can take $\pi^{(\delta)}(x)=x$ so
evidently \eqref{E3} holds true.
In terms of the definition of $b^{(\delta)}$ introduced in \eqref{E7},
it is easy to see  from (${\bf A}_1$) that for   any   $\theta\in(0,1/2)$   and $x,y\in\R^d,$
\begin{align*}
\big|b^{(\delta)}(x)-b^{(\delta)}(y)\big|&=\frac{\big|b(x)-b(y)+\delta^{\theta}(b(x)\varphi(|y|)-b(y)\varphi(|x|))\big|}{1+\delta^\theta(\varphi(|x|)+\varphi(|y|))+\delta^{2\theta}\varphi(|x|)\varphi(|y|)}\\
&\le\frac{(L_1\vee L_2 ) \big(  1+ \varphi(|x|)+\varphi(|y|)   +\delta^{\theta}  \varphi(|x|)\varphi(|y|)\big) |x-y|}{1+\delta^\theta(\varphi(|x|)+\varphi(|y|))+\delta^{2\theta}\varphi(|x|) \varphi(|y|)}\I_{\{x,y\in B_R\}}\\
&\quad+\frac{(L_1\vee L_2 )\delta^{-\theta}\big(  \delta^{\theta}+\delta^{\theta}(\varphi(|x|)+\varphi(|y|))   +\delta^{2\theta}  \varphi(|x|)\varphi(|y|)\big) |x-y|}{1+\delta^\theta(\varphi(|x|)+\varphi(|y|))+\delta^{2\theta}\varphi(|x|) \varphi(|y|)}\I_{\{x\in B_R^c\}\cup\{y\in B_R^c\}}.
\end{align*}
Therefore, (${\bf H}_1$) holds true for $\pi^{(\delta)}(x)=x $,  $C_R= (L_1\vee L_2 )(1+\varphi(R))^2$, and  $K_R=L_1\vee L_2.$

Again, by invoking  the definition of $b^{(\delta)}$, we derive from (${\bf A}_2$) that  for all  $x,y\in\R^d$  with $x\in B_R^c$  or $y\in B_R^c$,
\begin{align*}
\<x&-y,b^{(\delta)}(x)-b^{(\delta)}(y)\>\\
&=\frac{\<x-y,b(x)-b(y)+\delta^{\theta}(b(x)\varphi(|y|)-b(y)\varphi(|x|))\>}{1+\delta^\theta(\varphi(|x|)+\varphi(|y|))+\delta^{2\theta}\varphi(|x|) \varphi(|y|)}\\
&\le -\frac{\big( (L_3-L_4\delta^\theta)(1+\varphi(|x|)+\varphi(|y|)) +L_5\delta^{\theta}  \varphi(|x|)\varphi(|y|) \big) |x-y|^2 }{1+\delta^\theta(\varphi(|x|)+\varphi(|y|))+\delta^{2\theta}\varphi(|x|) \varphi(|y|)}\\
&\le -\big(  (L_3/2)\wedge L_5\big)\delta^{-\theta} \times\frac{ \big (\delta^\theta+\delta^\theta(\varphi(|x|)+\varphi(|y|))  + \delta^{2\theta}  \varphi(|x|)\varphi(|y|)\big)  |x-y|^2 }{1+\delta^\theta(\varphi(|x|)+\varphi(|y|))+\delta^{2\theta}\varphi(|x|) \varphi(|y|)},
\end{align*}
where in the second inequality we used the fact that $L_4\delta^\theta\le L_3/2$ for any  $\delta\in(0,\delta_0]$.
Note  that, for any  $\delta\in(0,\delta_0]$
and $r>0$, $(\delta^\theta+r)/(1+r)\ge (\delta^\theta+\delta^\theta r)/(1+r)=\delta^\theta$.  Whereafter, we infer that   for all   $x,y\in\R^d$ with $x\in B_R^c$  or $y\in B_R^c$,
\begin{align*}
\<x-y,b^{(\delta)}(x)-b^{(\delta)}(y)\>\le  -\big(  (L_3/2)\wedge L_5\big)|x-y|^2.
\end{align*}
Consequently, (${\bf H}_2$) for $K_R^*=(L_3/2)\wedge L_5$  is verifiable.

Below,    an illustrative example  is set to show that both $({\bf A}_1)$ and $({\bf A}_2)$ can be
validated.

\begin{example}\label{exa} Let $U:\R^d\to\R$ be the double well potential, which is defined by $U(x)=\frac{1}{4}|x|^4-\frac{1}{2}|x|^2$,   and set $b(x):=-\nn U(x)$ for $x\in\R^d.$ It is easy to see that
$b(x)=-|x|^2x+x.$ Via the triangle inequality, there exists a   constant $c_1>0 $ such that
\begin{equation}\label{ER5}
\begin{split}
|b(x)-b(y)|&\le |x-y|+\big||x|^2x-|y|^2y\big|\\
&\le|x-y|+|x|^2|x-y|+\big||x|^2-|y|^2\big|\cdot|y|\\
&\le c_1\big(1+|x|^2+|y|^2)|x-y|,\quad x,y\in\R^d,
\end{split}
\end{equation}
and moreover there is a constant   $c_2>0 $ such that
\begin{align*}
\big|b(x)|y|^2-b(y)|x|^2\big|&=\big|(-|x|^2x+x)|y|^2-(-|y|^2y+y)|x|^2\big|\\
&\le |x|^2|y|^2 |x -y  |+\big|x|y|^2-y|x|^2\big|\\
&\le  \big( |x|^2|y|^2 +|x|(|y| +|x| )+ |x|^2 \big) |x-y|\\
&\le c_2\big(1+|x|^2+|y|^2+|x|^2|y|^2\big)|x-y|,\quad x,y\in\R^d.
\end{align*}
Hence, we conclude that $({\bf A}_1)$ with
$\varphi(r)=r^2$ is satisfied.

Next, by invoking Newton-Leibniz's  formula, it follows that for any $x,y\in\R^d,$
\begin{align}
\<x-y,b(x)-b(y)\>&=|x-y|^2-\<x-y, |x|^2x -|y|^2y \>\nonumber\\
&=|x-y|^2-\int_0^1\frac{\d}{\d s}\<x-y,  |y+s(x-y)|^2(y+s(x-y))  \>\,\d s\nonumber\\
&=|x-y|^2-2\int_0^1  \<x-y,    y +s(x-y)\>^2 \,\d s\nonumber\\
&\quad-\int_0^1|y+s(x-y)|^2\,\d s|x-y|^2\label{EE6-}\\
&\le |x-y|^2-\big(|y|^2+\<y,x-y\>+ |x-y|^2/3\big)|x-y|^2\nonumber\\
&=\Big(1-\frac{1}{3}\big(|x|^2+|y|^2-\<x,y\>\big)\Big)|x-y|^2\nonumber\\
&\le\Big(1-\frac{1}{6}\big(|x|^2+|y|^2\big)\Big)|x-y|^2\nonumber ,
\end{align}
where in the first inequality we dropped the first integral (which is non-negative) in the third identity and in the last inequality we used the fact that
 $$|x|^2+|y|^2-\<x,y\>\ge \frac{1}{2}\left(|x|^2+|y|^2\right);$$ moreover, for  $x,y\in\R^d,$
\begin{align*}
\<x-y, b(x)|y|^2-b(y)|x|^2\>
&=\<x-y,x|y|^2-y|x|^2\>-|x-y|^2|x|^2|y|^2\\
&=|x-y|^2|y|^2+\<x-y,y\>(|y|^2-|x|^2)-|x-y|^2|x|^2|y|^2\\
&\le (|x||y|+2|y|^2)|x-y|^2-|x|^2|y|^2|x-y|^2\\
&\le \left(\frac{1}{2}|x|^2+\frac{5}{2}|y|^2\right)|x-y|^2-|x|^2|y|^2|x-y|^2.
\end{align*}
As a consequence, (${\bf A}_2$) is also  reachable  with $\varphi(r)=r^2$.

We note that, by following the preceding procedure, both (${\bf A}_1$) and (${\bf A}_2$) are still valid for $b(x)=-\nn V(x)$ with $V(x)=|x|^{2\beta}$ for $\beta>1,$ which has been investigated in \cite{LMW}.
\end{example}

\subsection{The truncated EM scheme}\label{sec2.3} In this subsection,
we intend to show that the truncated EM algorithm, which was initiated in   \cite{Mao}, can also be incorporated in \eqref{E2}.  Likewise, we assume that
$b$ is   locally Lipschitz continuous and partially dissipative.  More precisely, we suppose  that
\begin{enumerate}
\item[(${\bf C}_1$)]  for any $r>0,$ there exists   a strictly increasing and continuous function $\varphi:\R_+\to\R_+$ such that for all $x,y\in B_r$,
\begin{align*}
|b(x)-b(y)|\le  \varphi(r)|x-y|;
\end{align*}

\item[(${\bf C}_2$)]there exist constants $R,K_R>0$ such that for all $x,y\in\R^d$ with $x\in B_R^c$  or $y\in B_R^c$,
\begin{align*}
\<x-y,b(x)-b(y)\>\le  -K_R |x-y|^2.
\end{align*}
\end{enumerate}

We still take $b$ given in Example \ref{exa} as a typical candidate.  From \eqref{ER5}, we conclude that (${\bf C}_1$) holds true with  $\varphi(r)=c_1(1+2r^2),$ which obviously is strictly increasing and  continuous on the interval $[0,\8)$. Moreover, with the aid of \eqref{EE6-}, it is easy to see that
(${\bf C}_2$) is valid for some constant $R>0.$

In spirit to $b^{(\delta)}$ defined in \eqref{E7}, it consists essentially in modifying   the original drift $b$ so the   amended version
can be dictated.  To this end,  we define the following truncation function:  for any $\theta\in(0,1/2)$ and $\delta>0$,
 \begin{align}\label{ER3}
\pi^{(\delta)}(x)=  \frac{1}{|x|}\big(|x|\wedge \varphi^{-1}(\delta^{-\theta})\big)x\I_{\{|x|>0\}},\quad x\in\R^d,
\end{align}
where $\varphi^{-1}$ means  the inverse of $\varphi$. With the truncation function $\pi^{(\delta)}$ above at hand,
the truncated EM scheme associated with  \eqref{E1} can be presented in the form below: for any $\delta>0$ and integer $n\ge0,$
\begin{equation}\label{ER2}
X_{(n+1)\delta}^{\delta}=\pi^{(\delta)}(X_{n\delta}^\delta)+b(\pi^{(\delta)}(X_{n\delta}^\delta))\delta+\delta^{1/2} \xi_{n+1},
\end{equation}
in which  $(\xi_n)_{n\ge1}$ are $d$-dimensional standard Gaussian random variables.
As a result, by choosing $b^{(\delta)}=b$ in \eqref{E2},
 we reproduce the variant \eqref{ER2} of the truncated EM scheme explored initially in \cite{Mao}. In literature, the truncated EM scheme is also termed as  the projected EM method; see, for instance, \cite{BIK}.

To begin, we show that the truncation mapping defined in \eqref{ER3} is contractive. Indeed,
by the Cauchy-Schwarz inequality,
note that for any $\delta>0$ and $x,y\in\R^d$,
\begin{equation}\label{EW6}
\begin{split}
 |x-y|^2-|\pi^{(\delta)}(x)-\pi^{(\delta)}(y) |^2
&\ge |x|^2-\big(|x|\wedge \varphi^{-1}(\delta^{-\theta})\big)^2 +|y|^2-\big(|y|\wedge \varphi^{-1}(\delta^{-\theta})\big)^2\\
&\quad-2\big(|x|\cdot|y|-(|x|\wedge \varphi^{-1}(\delta^{-\theta}))(|y|\wedge \varphi^{-1}(\delta^{-\theta})\big)\\
&=:\Lambda(x,y,\delta).
\end{split}
\end{equation}
On the one hand,  the contractive property \eqref{E3} is fulfilled due to  $\Lambda(x,y,\delta)=0$ for any $x,y\in\R^d$ with $|x|,|y|\le \varphi^{-1}(\delta^{-\theta})$. On the other hand, a direct calculation shows that for $x,y\in\R^d$ with
$|y|\ge \varphi^{-1}(\delta^{-\theta})$,
\begin{align*}
\Lambda(x,y,\delta)=\big(|y|-(|x|\vee \varphi^{-1}(\delta^{-\theta}))\big(|y| +|x|\vee \varphi^{-1}(\delta^{-\theta})-2|x|\big).
\end{align*}
In particular,
\begin{equation*}
\Lambda(x,y,\delta)=
\begin{cases}
\big( |y|-   \varphi^{-1}(\delta^{-\theta}) \big)\big(|y| +\varphi^{-1}(\delta^{-\theta})-2|x|\big)\ge0,\quad |x| \le \varphi^{-1}(\delta^{-\theta}), |y| \ge \varphi^{-1}(\delta^{-\theta}), \\
(|x|-|y|)^2\ge0,~~~~~~~~~~~~~~~~~~~~~~~~~~~~~~~~~~~~~~~~~|x|,|y| \ge \varphi^{-1}(\delta^{-\theta}).
\end{cases}
\end{equation*}
Then, by taking advantage of the symmetry concerned with the variables $x$ and $y$, we conclude that \eqref{E3} is also valid for any $x,y\in\R^d$
with
$|x| \ge \varphi^{-1}(\delta^{-\theta})$.
In brief, the   contractive property \eqref{E3} is provable  for $\pi^{\delta}$ defined in
\eqref{ER3}.

Notice from $({\bf C_1})$ and the contractive property of $\pi^{{\delta}}$ that for any $\delta>0$, and $x,y\in\R^d,$
\begin{align*}
\big|b(\pi^{(\delta)}(x))-b(\pi^{(\delta)}(y))\big|&=\big|b(\pi^{(\delta)}(x))-b(\pi^{(\delta)}(y))\big|\big(\I_{\{x,y\in B_R\}}+\I_{\{x\in B_R^c\}\cup\{y\in B_R^c\}}\big)\\
&\le  \big(\varphi(R)\I_{\{x,y\in B_R\}}+ \delta^{-\theta}\I_{\{x\in B_R^c\} \cup\{y\in B_R^c\}}\big)|x-y|.
\end{align*}
Hence, we infer that Assumption $({\bf H}_1)$ holds true with $ C_R=\varphi(R)$,
$K_R=1,$ and any $\delta>0.$

Below, we set $\delta_0:=\varphi(R)^{-\frac{1}{\theta}}\wedge1$. By invoking  the strictly increasing property of $\varphi(\cdot)$,
it is ready to see  that for any $x\in\R^d$ with $x\in B^c_R$ and $\delta\in(0,\delta_0]$,
\begin{align*}
|\pi^{(\delta)}(x)|=|x|\wedge \varphi^{-1}(\delta^{-\theta})\ge|x|\wedge \varphi^{-1}(\delta^{-\theta}_0) \ge R.
\end{align*}
Based on this, we thus deduce from $({\bf C}_2)$ that for any $x,y\in\R^d$ with $x\in B_R^c$ or $y\in B_R^c,$
\begin{align*}
\big\< \pi^{(\delta)}(x)-\pi^{(\delta)}(y), b(\pi^{(\delta)}(x))-b(\pi^{(\delta)}(y))\big\>\le-K_R \big|\pi^{(\delta)}(x)-\pi^{(\delta)}(y)\big|^2.
\end{align*}
This obviously  yields that $({\bf H}_2)$ is valid.

\section{Proof of Theorem \ref{thm}}\label{sec3}
In this section, we aim to implement the proof of Theorem \ref{thm}. Before that, a series of preliminary work need to be done. In the first place, we introduce some notations.
Since, for each fixed $x\in\R^d$, $\nu_x(\d z)$ is absolutely continuous with respect to $\nu(\d z)$, the  Radon–Nikodym derivative, written as $\rho(x,z)$, exists so
 \begin{align}\label{E6}
 \rho(x,z):=\frac{\nu_x(\d z)}{\nu(\d z)} =\frac{(2\pi)^{-{d}/{2}}\big(\e^{-{|z|^2}/{2}}\wedge \e^{-{|z-x|^2}/{2}}\big)}{(2\pi)^{-{d}/{2}}\e^{-{|z|^2}/{2}}}\in(0,1],\quad x,z\in\R^d.
 \end{align}
Furthermore,   we set  for a threshold $\kk>0 $,
 $$(x)_\kk:=\Big(1\wedge \frac{\kk}{|x|}\Big)x\I_{\{|x|\neq 0\}},\quad x\in\R^d.$$
Let $(U_n)_{n\ge1}$ be a sequence of i.i.d.  random variables,  carried on the probability space $(\OO,\mathscr F, \P)$,
distributed as  uniform distributions on  $[0,1] $, and  independent of $(\xi_n)_{n\ge1}. $

With the preceding notations,
 we define the following  iteration: for $n\ge0$,
\begin{equation}\label{E4}
\begin{cases}
X_{(n+1)\delta}^{\delta}=\hat {X}_{n\delta}^{\delta}+\delta^{1/2}\xi_{n+1}\\
 Y_{(n+1)\delta}^{\delta}=\hat {Y}_{n\delta}^{\delta}+\delta^{1/2}\Big\{\big(\xi_{n+1}+\delta^{-{1}/{2}}(\hat {Z}_{n\delta}^{\delta})_\kk\big)\I_{\{U_{n+1}\le \frac{1}{2}\rho(-\delta^{-{1}/{2}} (\hat {Z}_{n\delta}^{\delta})_\kk,\xi_{n+1})\}}\\
 \qquad\qquad\qquad\qquad\quad+\big(\xi_{n+1}-\delta^{-{1}/{2}} (\hat {Z}_{n\delta}^{\delta})_\kk\big)\\
  \qquad\qquad\qquad\qquad\quad\times\I_{\{ \frac{1}{2}\rho(-\delta^{-{1}/{2}}(\hat {Z}_{n\delta}^{\delta})_\kk,\xi_{n+1})\le U_{n+1}\le \frac{1}{2}(\rho(-\delta^{-{1}/{2}} (\hat {Z}_{n\delta}^{\delta})_\kk,\xi_{n+1})+\rho(\delta^{-{1}/{2}} (\hat {Z}_{n\delta}^{\delta})_\kk,\xi_{n+1}))\}}\\
 \qquad\qquad\qquad\qquad\quad+\xi_{n+1}\I_{\{  \frac{1}{2}(\rho(-\delta^{-{1}/{2}} (\hat {Z}_{n\delta}^{\delta})_\kk,\xi_{n+1})+\rho(\delta^{-{1}/{2}} (\hat {Z}_{n\delta}^{\delta})_\kk,\xi_{n+1}))\le U_{n+1}\le 1\}} \Big\},
\end{cases}
\end{equation}
 where $\hat {Z}_{n\delta}^{\delta}:=\hat {X}_{n\delta}^{\delta}-\hat {Y}_{n\delta}^{\delta}$ with
\begin{align}\label{EE8}
\hat {X}_{n\delta}^{\delta}:=\pi^{(\delta)}(X_{n\delta}^\delta)+b^{(\delta)}(\pi^{(\delta)}(X_{n\delta}^\delta))\delta~\mbox{ and } ~\hat {Y}_{n\delta}^{\delta}:=\pi^{(\delta)}(Y_{n\delta}^\delta)+b^{(\delta)}(\pi^{(\delta)}(Y_{n\delta}^\delta))\delta.
\end{align}

Below, we elaborate the underlying intuition concerning the construction of the  coupling process $(X_{n\delta}^\delta, Y_{n\delta}^\delta)_{n\ge0}$,
which will be examined in Lemma \ref{coupling}.

\begin{remark}
To describe clearly the intuitive idea lay in the construction of the  coupling process,    we consider the one-step version of \eqref{E4}, i.e., for all $x,y\in\R^d,$
\begin{equation}\label{E4*}
\begin{cases}
 X_\delta^{\delta}=\hat x^{\delta}+\delta^{1/2}\xi_1\\
 Y_\delta^{\delta}=\hat y^{\delta}+\delta^{1/2}\Big\{\big(\xi_1+\delta^{-{1}/{2}}(\hat z^{\delta})_\kk\big)\I_{\{U_1\le \frac{1}{2}\rho(-\delta^{-{1}/{2}} (\hat z^{\delta})_\kk,\xi_1)\}}+\big(\xi_1-\delta^{-{1}/{2}} (\hat z^{\delta})_\kk\big)\\
 \qquad\qquad\qquad\quad\times\I_{\{ \frac{1}{2}\rho(-\delta^{-{1}/{2}}(\hat z^{\delta})_\kk,\xi_1)\le U_1\le \frac{1}{2}(\rho(-\delta^{-{1}/{2}} (\hat z^{\delta})_\kk,\xi_1)+\rho(\delta^{-{1}/{2}} (\hat z^{\delta})_\kk,\xi_1))\}}\\
 \qquad\qquad\qquad\quad+\xi_1\I_{\{  \frac{1}{2}(\rho(-\delta^{-{1}/{2}} (\hat z^{\delta})_\kk,\xi_1)+\rho(\delta^{-{1}/{2}} (\hat z^\delta)_\kk,\xi_1))\le U_1\le 1\}} \Big\},
\end{cases}
\end{equation}
where  $\hat z^{\delta}:=\hat x^{\delta}-\hat y^{\delta}$ with $\hat x^\delta :=\pi^{(\delta)}(x)+b^{(\delta)}(\pi^{(\delta)}(x))\delta$ for $x\in\R^d$.
The coupling given in \eqref{E4*} is inspired essentially  by the refined basic coupling proposed in \cite{LW} for  SDEs driven by  additive L\'{e}vy noises.
More precisely, for the given initial value $(x,y)$, the random variable $\xi_1$ is drawn to describe the fluctuation of $X_\delta^{\delta}$; In case of $|\hat z^\delta|\le \kk,$ with half of the maximum probability (i.e., $\frac{1}{2}\nu_{-\delta^{-{1}/{2}}\hat z^\delta}(\d z)$), $X_\delta^{\delta}$ and $Y_\delta^{\delta}$
 meet together; with the other half  (i.e., $\frac{1}{2}\nu_{\delta^{-{1}/{2}}\hat z^\delta}(\d z)$), the distance between $X_\delta^{\delta}$ and $Y_\delta^{\delta}$ is doubled with contrast to the original distance, which also plays a crucial role in guaranteeing the marginal property; with the remaining mass (i.e., $ \nu(\d z)- \frac{1}{2}\nu_{-\delta^{-{1}/{2}}\hat z^\delta}(\d z)-\frac{1}{2}\nu_{\delta^{-{1}/{2}}\hat z^\delta}(\d z)$), $X_\delta^{\delta}$ and $Y_\delta^{\delta}$ move forward synchronously. Additionally, we would like to emphasize that
 the  involvement of the threshold $\kk$ is indispensable. In particular, it can provide some uniform lower bound of the positive coupling probability; see the estimate \eqref{EE} below for more details.
\end{remark}

To begin with, we  show that $(X_{n\delta}^\delta, Y_{n\delta}^\delta)_{n\ge0}$
determined by \eqref{E4} is a coupling process of $(X_{n\delta}^\delta)_{n\ge0}$,
i.e., both $(Y_{n\delta}^\delta)_{n\ge0}$ and $(X_{n\delta}^\delta)_{n\ge0}$ have the same transition probabilities if they are  initialized from the same starting positions.

\begin{lemma}\label{coupling}
For  $\delta,\kk>0$,
$(X_{n\delta}^{\delta},Y_{n\delta}^{\delta})_{n\ge0}$ is a coupling process of $(X_{n\delta}^{\delta})_{n\ge0}$.
\end{lemma}
\begin{proof}
Let  $\bar\P$   be the law of $(X_{n\delta}^\delta,Y_{n\delta}^\delta)_{n\ge0}$   and $\bar \E$    be the corresponding expectation  operator under $\bar\P$.
To show that $(X_{n\delta}^{\delta},Y_{n\delta}^{\delta})_{n\ge0}$ is a coupling  process of $(X_{n\delta}^{\delta})_{n\ge0}$, it is sufficient to verify that for all  $f\in B_b(\R^d)$ and any integer $n\ge1$,
\begin{align}\label{ET}
\bar\E\big( f(Y_{n\delta}^\delta)\big|\big(X_{(n-1)\delta}^\delta,Y_{(n-1)\delta}^\delta\big)\big)=\E\big( f(X_{n\delta}^{\delta})\big|X_{n-1}^\delta\big).
\end{align}
Recall that $(\xi_n)_{n\ge1}$ (resp. $(U_n)_{n\ge1}$) are i.i.d random variables
and $(\xi_n)_{n\ge1}$ is independent of $(U_n)_{n\ge1}$. Hence,
in order to achieve \eqref{ET}, via an inductive argument,    it is essential to verify
\begin{align} \label{E5}
\bar\E^{(x,y)}f(Y_\delta^\delta):=\bar\E \big( f(Y_\delta^\delta)\big|(X_0^\delta,Y_0^\delta)=(x,y)\big)  =\E \big( f(X_\delta^{\delta})\big|X_0^{\delta}=y\big).
\end{align}
Indeed, since $\xi_1$ with the law $\nu(\d z)$ is  independent of the uniformly distributed random variable $U_1$   on $[0,1],$ we derive that
\begin{align*}
\bar\E^{(x,y)}  f(Y_\delta^\delta)&=\frac{1}{2}\int_{\R^d}f\big(\hat y^{\delta}+\delta^{1/2}\big(u+\delta^{-{1}/{2}}(\hat z^{\delta})_\kk\big)\big) \,\nu_{-\delta^{-{1}/{2}}(\hat z^{\delta})_\kk}(\d u)\\
&\quad+\frac{1}{2}\int_{\R^d}f\big(\hat y^{\delta}+\delta^{1/2}\big(u-\delta^{-{1}/{2}}( \hat z^{\delta})_\kk\big)\big) \,\nu_{\delta^{-{1}/{2}}(\hat z^{\delta})_\kk}(\d u)\\
&\quad+\int_{\R^d}f\big(\hat y^{\delta}+\delta^{1/2} u\big)\Big(\nu(\d u)-\frac{1}{2} \nu_{-\delta^{-{1}/{2}} (\hat z^{\delta})_\kk}(\d u) -\frac{1}{2} \nu_{\delta^{-{1}/{2}}(\hat z^{\delta})_\kk} \nu(\d u)\Big).
\end{align*}
For the first integral and the second integral above, via change of variables $u\rightarrow u-\delta^{-{1}/{2}}(\hat z^{\delta})_\kk$ and $u\rightarrow u+\delta^{-{1}/{2}}(\hat z^{\delta})_\kk$    respectively, in addition to $\nu_{-x}(\d (u-x))=\nu_x(\d u)$ for $x\in\R^d$, we find that
\begin{align*}
\bar\E^{(x,y)}  f(Y_\delta^\delta)&=\frac{1}{2}\int_{\R^d}f\big(\hat y^{\delta}+\delta^{1/2} u\big)\,\nu_{\delta^{-{1}/{2}}(\hat z^{\delta})_\kk}(\d u)\\
&\quad+\frac{1}{2}\int_{\R^d}f\big(\hat y^{\delta}+\delta^{1/2} u\big)\,\nu_{-\delta^{-{1}/{2}}(\hat z^{\delta})_\kk}(\d u)\\
&\quad+\int_{\R^d}f\big(\hat y^{\delta}+\delta^{1/2} u\big)\Big(\nu(\d u
)-\frac{1}{2} \nu_{-\delta^{-{1}/{2}}(\hat z^{\delta})_\kk}(\d u)-\frac{1}{2} \nu_{\delta^{-{1}/{2}}(\hat z^{\delta})_\kk} \nu(\d u)\Big)\\
&=\int_{\R^d}f\big(\hat y^{\delta}+\delta^{1/2} u\big)\nu(\d u
).
\end{align*}
Therefore, \eqref{E5} follows directly.
\end{proof}

\begin{lemma}\label{lem1}
Assume \eqref{E3},
$({\bf H}_1)$ and $({\bf H}_2)$ hold. Then, for any $\delta\in(0,\delta_1]$ with $\delta_1$ being given in \eqref{EE6*}  and integer $n\ge1,$
\begin{equation}\label{E9}
\begin{split}
&\bar\E\big( |X_{n\delta}^\delta-Y_{n\delta}^\delta| \big|(X_{(n-1)\delta}^\delta,Y_{(n-1)\delta}^\delta)\big)-|X_{(n-1)\delta}^\delta-Y_{(n-1)\delta}^\delta|\\
&\le  \Big( C_R\delta  \I_{\{|X_{(n-1)\delta}^\delta-Y_{(n-1)\delta}^\delta|\le 2R\}}- \frac{1}{2}K_R^*\delta \I_{\{|X_{(n-1)\delta}^\delta-Y_{(n-1)\delta}^\delta|>2R\}}\Big)|X_{(n-1)\delta}^\delta-Y_{(n-1)\delta}^\delta|.
\end{split}
\end{equation}
Moreover, for any $\delta\in(0,\delta_1]$, $\kk>0$, and integer $n\ge1,$
\begin{align}\label{EE}
\bar\E\big(\I_{\{ X_{n\delta}^{\delta }=Y_{n\delta}^{\delta}\}}\big|(X_{(n-1)\delta}^\delta,Y_{(n-1)\delta}^\delta)\big)\ge \frac{1}{2}J\big(\kk\delta^{-{1}/{2}}\big)
\end{align}
if
$|X_{(n-1)\delta}^\delta-Y_{(n-1)\delta}^\delta|\le \kk/(1+C_R)$.
\end{lemma}

Before we proceed, we make some remarks on Lemma \ref{lem1}.
\begin{remark}
Inequality \eqref{E9}
characterizes the drift condition for the scheme \eqref{EW}. Particularly, it shows that the algorithm under investigation
does not need to be  dissipative in the short distance but dissipative merely in the long distance.
Inequality \eqref{EE} reveals that, for any fixed $n\ge1,$
  $X_{n\delta}^\delta$ coincides with $Y_{n\delta}^\delta$ with positive probability when
 the distance between the previous values $X_{(n-1)\delta}^\delta$ and $Y_{(n-1)\delta}^\delta$ is small.
  Furthermore, we would like to stress that the incorporation
  of the  threshold $\kk$ is indispensable. Since the support of Gaussian measures is  $\R^d$, it  would be
  quite natural to take the threshold $\kk=\8.$ Unfortunately,
 once we take $\kk=\8,$ the right hand side of \eqref{EE} should be  reformulated as
 $\frac{1}{2}J(\delta^{-{1}/{2}} (1+C_R)|X_{(n-1)\delta}^\delta-Y_{(n-1)\delta}^\delta|)$,
 see the proof of \eqref{EE*} in Lemma \ref{lem1}.
  However, this quantity might approach zero when $\delta$ goes to zero, even if the distance $|X_{(n-1)\delta}^\delta-Y_{(n-1)\delta}^\delta|$ is sufficiently small.
  Whereas, by choosing appropriately the critical point $\kk$
(e.g., $\kk=\delta^{{1}/{2}}$),
the number on the right hand side of \eqref{EE} is strictly positive even if the step size decays to zero.
 This shows the importance of setting the threshold $\kk < \infty$.
\end{remark}

Now we move to  finish the proof of Lemma \ref{lem1}.

\begin{proof}[Proof of Lemma $\ref{lem1}$]
In order to obtain \eqref{E9} and \eqref{EE}, it suffices to show that for any $\delta\in(0,\delta_1]$, and $x,y\in\R^d,$
\begin{equation}\label{E9*}
\begin{split}
\bar\E^{(x,y)}\big|X_\delta^\delta-Y_\delta^\delta\big|:&=\bar\E\big(\big|X_\delta^\delta-Y_\delta^\delta\big|\big|(X_0^\delta,Y_0^\delta)=(x,y)\big)\\
&\le  \Big(1+C_R\delta  \I_{\{|x-y|\le 2R\}}- \frac{1}{2}K_R^*\delta \I_{\{|x-y|>2R\}}\Big)|x-y|,
\end{split}
\end{equation}
and for any $\delta\in(0,\delta_1]$, $\kk>0$, and $x,y\in\R^d$ with $|x-y|\le\kk/(1+C_R)$,
\begin{align}\label{EE*}
 \bar\P^{(x,y)}\big(X_\delta^{\delta}
 = Y_\delta^{\delta}\big):=\bar\E^{(x,y)}\I_{\{X_\delta^{\delta}=
 Y_\delta^{\delta}\}}
 \ge \frac{1}{2}J\big(\kk\delta^{-{1}/{2}} \big).
\end{align}

In the sequel,  we shall fix $(X_0^\delta,Y_0^\delta)=(x,y)$ and
stipulate $\delta\in(0,\delta_1]$ so that
\begin{align}\label{WW}
 1-2K_R^*\delta\ge0 \quad \mbox{ and }  \quad K_R^*-K_R^2\delta^{1-2\theta}\ge0.
\end{align}
From \eqref{E4*}, it is easy to see that
\begin{equation}\label{E13}
\begin{split}
X_\delta^{\delta}-Y_\delta^{\delta}=&\hat z^{\delta}-(\hat z^{\delta})_\kk\I_{\{U_1\le \frac{1}{2}\rho(-\delta^{-{1}/{2}}(\hat z^{\delta})_\kk,\xi_1)\}}\\
&+ (\hat z^{\delta})_\kk\I_{\{ \frac{1}{2}\rho(-\delta^{-{1}/{2}}(\hat z^{\delta})_\kk,\xi_1)\le U_1\le \frac{1}{2}(\rho(-\delta^{-{1}/{2}}(\hat z^{\delta})_\kk,\xi_1)+\rho(\delta^{-{1}/{2}}(\hat z^{\delta})_\kk,\xi_1))\}}.
\end{split}
\end{equation}
Whence, we find that
\begin{align*}
\bar\E^{(x,y)}\big|X_\delta^\delta-Y_\delta^\delta\big|&= \frac{1}{2}\big|\hat z^{\delta}-(\hat z^{\delta})_\kk\big|\nu_{-\delta^{-{1}/{2}}(\hat z^{\delta})_\kk}(\R^d)+\frac{1}{2}\big|\hat z^{\delta}+(\hat z^{\delta})_\kk\big|\nu_{\delta^{-{1}/{2}}(\hat z^{\delta})_\kk}(\R^d)\\
&\quad+|\hat z^{\delta}|\Big(1-\frac{1}{2}\nu_{-\delta^{-{1}/{2}}(\hat z^{\delta})_\kk}(\R^d)-\frac{1}{2}\nu_{\delta^{-{1}/{2}}
(\hat z^{\delta})_\kk}(\R^d)\Big)\\
&=\frac{1}{2}|\hat z^{\delta}|\bigg(1-\frac{|\hat z^{\delta}|\wedge\kk}{|\hat z^{\delta}|}\I_{\{|\hat z^{\delta}|\neq 0\}}\bigg)\nu_{-\delta^{-{1}/{2}}(\hat z^{\delta})_\kk}(\R^d)\\
&\quad+\frac{1}{2}|\hat z^{\delta}|\bigg(1+\frac{|\hat z^{\delta}|\wedge\kk}{|\hat z^{\delta}|}\I_{\{|\hat z^{\delta}|\neq 0\}}\bigg)\nu_{\delta^{-{1}/{2}}(\hat z^{\delta})_\kk}(\R^d)\\
&\quad+|\hat z^{\delta}|\Big(1-\frac{1}{2}\nu_{-\delta^{-{1}/{2}}(\hat z^{\delta})_\kk}(\R^d)-\frac{1}{2}\nu_{\delta^{-{1}/{2}}(\hat z^{\delta})_\kk}(\R^d)\Big).
\end{align*}
This, besides the fact that $\nu_u(\R^d)=\nu_{-u}(\R^d)$ for all $u\in\R^d$,
yields
\begin{equation}\label{E12}
\bar\E^{(x,y)}\big|X_\delta^\delta-Y_\delta^\delta\big| =|\hat z^{\delta}|.
\end{equation}

By utilizing  \eqref{E3},  (${\bf H}_1$) as well as (${\bf H}_2$), we deduce  that
\begin{align*}
|\hat z^{\delta}|^2&=\big|\pi^{(\delta)}(x)-\pi^{(\delta)}(y)\big|^2+2\<\pi^{(\delta)}(x)-\pi^{(\delta)}(y),b^{(\delta)}(\pi^{(\delta)}(x)) -b^{(\delta)}(\pi^{(\delta)}(y))\>\delta\\
&\quad+\big|b^{(\delta)}(\pi^{(\delta)}(x)) -b^{(\delta)}(\pi^{(\delta)}(y))\big|^2\delta^2\\
&\le \big(1+ C_R\delta\big)^2|x-y|^2\I_{\{x,y\in B_R\}}\\
&\quad+\big((1-2K_R^*\delta)\big|\pi^{(\delta)}(x)-\pi^{(\delta)}(y)\big|^2+K_R^2\delta^{2(1- \theta)}|x-y|^2\big)\I_{\{x\in B_R^c\}\cup\{y\in B_R^c\}}.
\end{align*}
As a consequence,  making use  of \eqref{E3} again, along with \eqref{WW},  implies that
\begin{equation}\label{E0}
\begin{split}
|\hat z^{\delta}|^2
&\le \big(1+ C_R\delta\big)^2|x-y|^2\I_{\{x,y\in B_R\}}\\
&\quad+\big( 1- K_R^*\delta -\big(K_R^*-K_R^2\delta^{  1- 2\theta }\big)\delta \big)|x-y|^2\I_{\{x\in B_R^c\}\cup\{y\in B_R^c\}}\\
&\le \big(1+ C_R\delta\big)^2|x-y|^2\I_{\{x,y\in B_R\}} +\big( 1- K_R^*\delta  \big)|x-y|^2\I_{\{x\in B_R^c\}\cup\{y\in B_R^c\}}.
\end{split}
\end{equation}
Subsequently, the preceding estimate, accompanying with  the inequality:
 $(1-r)^{\frac{1}{2}}\le 1- r/2$ for all $ r\in[0,1], $
enables us to derive from \eqref{E12} that
\begin{align*}
\bar\E^{(x,y)}\big|X_\delta^\delta-Y_\delta^\delta\big|&\le  (1+ C_R\delta) |x-y| \I_{\{x,y\in B_R\}} +   ( 1- K_R^*\delta/2   )|x-y| \I_{\{x\in B_R^c\}\cup\{y\in B_R^c\}}\\
&\le
\Big(1+C_R\delta \I_{\{|x-y|\le 2R\}}- \frac{1}{2}K_R^*\delta \I_{\{|x-y|>2R\}}\Big)|x-y|.
\end{align*}
As a result, \eqref{E9*} follows immediately.

By invoking \eqref{E0} once more, we directly have for $\delta\in(0,1],$
\begin{align*}
|\hat z^{\delta}|\le (1+C_R\delta)|x-y|\le (1+C_R)|x-y|.
\end{align*}
This thus implies that $\hat z^{\delta}-(\hat z^{\delta})_\kk=0$ for any $x,y\in\R^d$ with $|x-y|\le \kk/(1+C_R)$. Whereafter,
by taking advantage of    the fact that $[0,\8)\ni r\mapsto J(r)$ is non-increasing,
it follows from \eqref{E13} that for any $x,y\in\R^d$ with $|x-y|\le \kk/(1+C_R)$,
\begin{align*}
\bar\P^{(x,y)}\big( X_\delta^{\delta}
= Y_\delta^{\delta}\big) \ge \frac{1}{2}\nu_{-\delta^{-{1}/{2}} (\hat z^{\delta})_\kk}(\R^d) \ge \frac{1}{2}J\big(\kk\delta^{-{1}/{2}} \big).
\end{align*}
Therefore, the assertion \eqref{EE} is attainable.
\end{proof}

With Lemmas \ref{coupling} and \ref{lem1} at hand, we are in position to complete the proof of Theorem \ref{thm}.
\begin{proof}[Proof of Theorem $\ref{thm}$]
For the quantities $c_*,a>0$, defined in \eqref{EP6},
we define an auxiliary    function and an associated   metric function as below:
\begin{align*}
f(r) =1-\e^{-c_*r}+c_*\e^{-2Rc_*}r,\quad r\ge0;
\end{align*}
and
\begin{align*}
 \hat\rho_1(x,y)= a\I_{\{ x\neq y\}}+f(|x-y|),\quad x,y\in\R^d.
\end{align*}
Note that the metrics $\rho_1$ (defined in \eqref{EW1}) and $\hat \rho_1$ are comparable.
Indeed, since $c_* \e^{-2Rc_*} \le f'(r) \le c_*(1+ \e^{-2Rc_*})$ for all $r\ge0$, we see that for all  $x,y\in\R^d,$
$$
\min \{ a, c_* \e^{-2Rc_*} \} \rho_1(x,y) \le \hat\rho_1(x,y) \le \max \{ a, c_*(1+ \e^{-2Rc_*}) \}  \rho_1(x,y).
$$
Therefore,  it is sufficient to demonstrate that  for any $\delta\in(0, \delta_1]$, $x,y\in\R^d,$ and integer $n\ge0,$
\begin{align}\label{WW1}
\mathbb W_{\hat\rho_1}\big(\mathscr L_{X_{n\delta}^{\delta,x}},\mathscr L_{ X_{n\delta}^{\delta,y}}\big)\le  \e^{-\lambda_1 n\delta}\hat\rho_1(x,y).
\end{align}
From \eqref{WW1}, we can obtain \eqref{W} with $C_1^*: = \frac{\max \{ a, c_*(1+ \e^{-2Rc_*}) \} }{\min \{ a, c_* \e^{-2Rc_*} \}}$.

Trivially,  \eqref{WW1}
holds true for the case $x,y\in\R^d$ with $x=y\in\R^d$. So, in the  analysis below, we assume $(X_0^\delta,Y_0^\delta)=(x,y)$ with $x\neq y.$ From Lemma \ref{coupling}, in addition to the tower property of conditional expectations,  it is easy to see that
\begin{align*}
\mathbb W_{\hat\rho_1}\big(\mathscr L_{X_{n\delta}^{\delta,x}},\mathscr L_{ X_{n\delta}^{\delta,y}}\big)&\le  \bar\E^{(x,y)}{\hat\rho_1}(X_{n\delta}^{\delta}, Y_{n\delta}^{\delta})= \bar\E^{(x,y)}\Big(\bar\E^{(x,y)}\big({\hat\rho_1}(X_{n\delta}^{\delta}, Y_{n\delta}^{\delta})\big|\big(X_{(n-1)\delta}^\delta,Y_{(n-1)\delta}^\delta\big)\big)\Big).	
\end{align*}
Further, recall that  $(\xi_n)_{n\ge1}$ (resp. $(U_n)_{n\ge1}$) are i.i.d random variables
and $(\xi_n)_{n\ge1}$ is independent of $(U_n)_{n\ge1}$.
As long as  there exists a constant $\lambda \in(0,1)$ (whose exact value will be given at the end of proof) such that for any $\delta\in(0,\delta_1]$,
\begin{align}\label{EW7}
\bar\E^{(x,y)}\hat\rho_1(X_\delta^{\delta}, Y_\delta^{\delta})\le (1-\lambda \delta)\hat\rho_1(x,y),
\end{align}
then performing an inductive argument yields that
\begin{align*}
\bar\E^{(x,y)}\hat\rho_1(X_{n\delta}^{\delta}, Y_{n\delta}^{\delta})\le (1-\lambda\delta)^n\rho(x,y)\le \e^{-\lambda n\delta}\hat\rho_1(x,y),
\end{align*}
where the second inequality is available due to the basic inequality: $r^a\le \e^{a(r-1)}$ for all $r,a>0.$ Therefore, to achieve the desired assertion \eqref{WW1},  it remains  to establish  \eqref{EW7}.

In the subsequent context, we shall choose the threshold $\kk= \delta^{1/2}$.
According to the definition of $\hat\rho_1,$ we  have
\begin{align*}
\hat\rho_1(X_\delta^{\delta},Y_\delta^{\delta})-\hat\rho_1(x,y)&=-a\I_{\{|X_\delta^{\delta}-Y_\delta^{\delta}|=0\}}+f(|X_\delta^{\delta}-Y_\delta^{\delta}|)-f(|x-y|).
\end{align*}
Thus, since  for any $r>0,$
\begin{align*}
f'(r)=c_*\big(\e^{-c_*r}+\e^{-2c_*R}\big), \quad f''(r)=-c^2_*\e^{-c_*r}<0 \quad \mbox{ and } \quad  f'''(r)=c^3_*\e^{-c_*r}>0,
\end{align*}
  Taylor's expansion enables us to derive that
\begin{equation}\label{WE}
\begin{split}
 \bar\E^{(x,y)}\hat\rho_1\big( X_\delta^{\delta},Y_\delta^{\delta}\big)-\hat\rho_1(x,y)
&\le-a\bar\P^{(x,y)}\big(X_\delta^{\delta}=Y_\delta^{\delta}\big)\\
&\quad+c_*\big(\e^{-c_*|x-y|}+\e^{-2Rc_* }\big)\Big(\bar\E^{(x,y)}|X_\delta^{\delta}-Y_\delta^{\delta}| - |x-y| \Big)\\
&\quad-\frac{1}{2}c^2_*\e^{-c_*|x-y|}\\
&\qquad\times\bar\E^{(x,y)}\Big(\big(|X_\delta^{\delta}-Y_\delta^{\delta}| - |x-y|\big)^2\I_{\{|X_\delta^{\delta}-Y_\delta^{\delta}|<|x-y|\}}\Big)\\
&=:I_1(\delta,\kk,x,y)+I_2(\delta,\kk,x,y) -\frac{1}{2}c^2_*\e^{-c_*|x-y|}I_3(\delta,\kk,x,y).
\end{split}
\end{equation}
Hereinafter,
for three cases based on the range of values of $|x-y|$,
 we shall estimate terms $I_1$, $I_2$, and $I_3$, to reach \eqref{EW7}.

  Case (i): $x,y\in\R^d$ with $0<|x-y|\le  r_0:= \delta^{1/2}/(1+C_R)$. Obviously, we have $ \delta^{1/2}/(1+C_R)\le 2R$ due to $R,C_R\ge1$ and  $\delta\in(0, 1].$
In the present case, in view of  \eqref{EE} and $\kk=\delta^{1/2}$,
one  has
\begin{align*}
I_1(\delta,\kk,x,y)\le-\frac{1}{2}aJ(1).
\end{align*}
Next, by virtue of \eqref{E9} and $|x-y|\le \delta^{1/2}/(1+C_R)$, we find that for  $\delta\in(0,1],$
\begin{align*}
I_2(\delta,\kk,x,y)&\le  c_*C_R\delta\big(\e^{-c_*|x-y|}+\e^{-2Rc_* } \big) |x-y|\le 2c_*.
\end{align*}
Whereafter, taking the estimates on $I_1$ and $I_2$ into consideration,
as well as $I_3\ge0,$ yields
\begin{equation}\label{EW2}
\begin{split}
\bar\E^{(x,y)}\hat\rho_1(X_\delta^{\delta},Y_\delta^{\delta})-\hat\rho_1(x,y)&\le -\frac{1}{2}aJ(1)+ 2c_* \le-\frac{1}{4}aJ(1),
\end{split}
\end{equation}
where the second inequality is available  by taking advantage of the
definition of $a$ in \eqref{EP6}.
Furthermore, in the present case $|x-y| \leq 2R$,
by invoking the fact that $\max_{r\ge0}r\e^{- r}=1/\e$,
it is immediate to see that
 \begin{align}\label{W*}
\hat\rho_1(x,y)\le 1+a +2Rc_*\e^{-2Rc_*}\le a+2.
\end{align}
As a consequence, we deduce from \eqref{EW2} that for $\delta\in(0,1],$
\begin{align*}
\bar\E^{(x,y)}\hat\rho_1(X_\delta^{\delta},Y_\delta^{\delta})\le   (1-  \lambda_{11}\delta )\hat\rho_1(x,y),
\end{align*}
where $\lambda_{11}:=\frac{aJ(1)  }{4(a+2 )}\in(0,1)$.

 Case (ii): $x,y\in\R^d$ with $r_0\le |x-y|\le 2R$. Concerning such setting, notice from \eqref{E9}  that
\begin{equation}\label{EW3}
\begin{split}
 I_2(\delta,\kk,x,y)
 &\le c_*C_R\delta \big( \e^{-c_*|x-y|}+ \e^{-2Rc_* } \big)|x-y| \le 4 c_*RC_R\delta  \e^{-c_*|x-y|}.
\end{split}
\end{equation}
Next,  from  \eqref{E13} and $\kk=\delta^{1/2}$, in addition to the non-increasing property of the mapping $[0,\8)\ni r\mapsto J(r)$,
we infer  that
\begin{align*}
I_3(\delta,\kk,x,y)
&\ge \frac{1}{2}\big(|\hat z^{\delta}|-(\kk\wedge|\hat z^{\delta}|) - |x-y|\big)^2\I_{\{|\hat z^{\delta}|-(\kk\wedge|\hat z^{\delta}|)<|x-y|\}} \nu_{
-\delta^{-{1}/{2}}(\hat z^{\delta})_\kk}(\R^d)\\
&\ge\frac{1}{2}\big(|\hat z^{\delta}|-(\kk\wedge|\hat z^{\delta}|) - |x-y|\big)^2\I_{\{|\hat z^{\delta}|-(\kk\wedge|\hat z^{\delta}|)<|x-y|\}} J(1).
\end{align*}
In view of  \eqref{E0}, it is easy to see that
\begin{align*}
|\hat z^\delta|\le (1+C_R\delta)|x-y|\le |x-y|+2RC_R \delta.
\end{align*}
This enables us to obtain, since $\kk=\delta^{1/2}$, that
\begin{align*}
|\hat z^{\delta}|-(\kk\wedge|\hat z^{\delta}|) - |x-y|&=  - |x-y| \I_{\{|\hat z^{\delta}|\le \kk\}}+\big(|\hat z^{\delta}|- |x-y|- \delta^{1/2} \big)\I_{\{|\hat z^{\delta}|> \kk\}}\\
&\le - |x-y| \I_{\{|\hat z^{\delta}|\le \kk\}}+\big(2RC_R\delta -  \delta^{1/2} \big)\I_{\{|\hat z^{\delta}|> \kk\}}\\
&\le - |x-y| \I_{\{|\hat z^{\delta}|\le \kk\}}  - \frac{1}{2} \delta^{1/2}  \I_{\{|\hat z^{\delta}|> \kk\}}\\
&\le - \frac{\delta^{1/2}}{1+C_R} ,
\end{align*}
where the second inequality is valid due to $\delta\le (4RC_R )^{-2}$ for $\delta\in(0,\delta_1]$,
cf.\ \eqref{EE6*},
 and the last display is true by exploiting
$\delta^{1/2}/(1+C_R)=r_0 \le |x-y|$
and $C_R\ge1$.
Consequently, we arrive at
\begin{align}\label{EP}
I_3(\delta,\kk,x,y)
\ge \frac{J(1)\delta}{2(1+C_R)^2} .
\end{align}
This estimate, combining  \eqref{EW3} with $I_1\le0$ and the definition of
$c_*$ in \eqref{EP6},
yields
 \begin{align*}
\bar\E^{(x,y)}\hat\rho_1(X_\delta^{\delta},Y_\delta^{\delta})-\hat\rho_1(x,y)
 \le -
4RC_Rc_*\e^{-2Rc_*}\delta.
\end{align*}
  Accordingly, we obtain from \eqref{W*} that
 \begin{align*}
\bar\E^{(x,y)}\hat\rho_1(X_\delta^{\delta},Y_\delta^{\delta})
&\le   (1-\lambda_{12}\delta) \hat\rho_1(x,y),
\end{align*}
where $\lambda_{12}:=\frac{
4RC_Rc_*\e^{-2Rc_*} }{ a+2} \in(0,1)$ by noting that $2 C_R/( a+2)<1$
owing to $a>2 C_R $.

 Case (iii): $x,y\in\R^d$ with $ |x-y|\ge 2R$. 
  Note that \eqref{E9} implies that $\bar\E^{(x,y)}|X_\delta^{\delta}-Y_\delta^{\delta}| - |x-y| \le - \frac{1}{2}K_R^*\delta|x-y|$. Hence, using $\e^{-c_*|x-y|} \ge 0$ in the bound for $I_2$, and due to $I_1\le0$ and $I_3\ge0$, we deduce from \eqref{WE} that
\begin{align*}
 \bar \E^{(x,y)}\hat\rho_1(X_\delta^{\delta},Y_\delta^{\delta})-\hat\rho_1(x,y)&\le c_*\big(\e^{-c_*|x-y|}+\e^{-2Rc_* }\big)\Big(\bar\E^{(x,y)}|X_\delta^{\delta}-Y_\delta^{\delta}| - |x-y| \Big)\\
&\le -\frac{1}{2} K_R^*\delta c_* \e^{-2Rc_* } |x-y|.
\end{align*}
This
implies that
\begin{align*}
\bar\E^{(x,y)}\hat\rho_1(X_\delta^{\delta},Y_\delta^{\delta})\le (1-\lambda_{13}\delta )\hat\rho_1(x,y),
\end{align*}
where
$\lambda_{13} = \frac{ R c_*K_R^*\e^{-2Rc_* }  }{a+2}$.
Indeed,
 since for $x,y\in\R^d$ with $|x-y|\ge2R$,
$$\hat\rho_1(x,y) \le a +1+  c_* \e^{-2Rc_*}|x-y|\le \Big(\frac{a+1}{2R}+c_* \e^{-2Rc_*}\Big)|x-y|,$$
we observe that $-K_R^*\delta c_* \e^{-2Rc_* } |x-y| \le - \lambda_{13}\delta \hat \rho_1(x,y)$ for $|x-y|\ge 2R$ by making use of the fact that
\begin{align*}
 -\frac{1}{2} K_R^*  c_* \e^{-2Rc_* } |x-y|
 &\le  -\frac{  K_R^*  Rc_* \e^{-2Rc_* }}{ a+1 +2Rc_* \e^{-2Rc_*}}\hat \rho_1(x,y)
\end{align*}
and  $\max_{r\ge0}r\e^{- r}=1/\e<1$.
 Furthermore, we notice that $\lambda_{13} \in (0,1)$ since
$a>K_R^*$ and  $\max_{r\ge0}r\e^{- r}=1/\e<1$.
To summarize,
based on the analysis above,  \eqref{EW7} is available   by taking $\lambda=\lambda_{11}\wedge\lambda_{12}\wedge\lambda_{13}$.
\end{proof}

\section{Proof of Theorem \ref{add}}\label{section_add}
Before we start the proof of Theorem \ref{add}, we first claim that $(X_{n\delta}^\delta)_{n\ge0}$ satisfies the Lyapunov condition in the semigroup type.
\begin{lemma}\label{Lya}
Under  Assumptions of Theorem $\ref{add}$,   for any $\delta\in(0,\delta_2]$ and any integer $n\ge0,$
\begin{align}\label{T1}
\E\big(\big(1+\big|X_{(n+1)\delta}^\delta\big|^2\big)\big|X_{n\delta}^\delta\big)\le (1-K_R^*\delta/2)\big(1+|X_{n\delta}^\delta|^2\big)+c_0\delta,
\end{align}
where $c_0>0$ was defined
 in \eqref{EP3}.
\end{lemma}

\begin{proof}
Below, we shall fix $\delta\in(0,\delta_2]$ so that
\begin{align}\label{ER0}
1-2K_R^*\delta \ge0\quad \mbox{ and } \quad \frac{K_R^*}{K_R^2\delta^{1-2\theta}}-1\ge \vv_0:=2^{1-2\theta}-1.
\end{align}
From $({\bf H}_1)$ and $({\bf H}_2)$,
as well as
$\pi^{(\delta)}({\bf0})={\bf0}$ and \eqref{ER0},
we  deduce  that for any $x\in\R^d$,
\begin{equation}\label{ER*}
\begin{split}
\Lambda(x):&=\big|\pi^{(\delta)}(x)\big|^2+2\big\<\pi^{(\delta)}(x),b^{(\delta)}(\pi^{(\delta)}(x))\big\>\delta
+\big|b^{(\delta)}(\pi^{(\delta)}(x))\big|^2\delta^2\\
&\le \big(|x|^2+(1+2C_R+2C_R^2)|x|^2\delta  +3\big(b^{(\8)}_0\big)^2\delta\big)\I_{\{x\in B_R\}}\\
&\quad+\Big((1-2K_R^*\delta)\big|\pi^{(\delta)}(x)\big|^2+\frac{1}{2}K_R^*|x|^2\delta+(1+ \vv_0)K_R^2|x|^2\delta^{2(1-\theta)}\\
&\quad\quad\quad+\big(1+1/\vv_0+2/K_R^*\big)\big(b^{(\8)}_0\big)^2\delta\Big)\I_{\{x\in B_R^c\}}\\
&\le\big((1-K_R^*\delta/2)|x|^2+c_1\delta\big)\I_{\{x\in B_R\}} \\
&\quad+\Big((1-2K_R^*\delta)\big|\pi^{(\delta)}(x)\big|^2+\frac{1}{2}K_R^*|x|^2\delta+(1+ \vv_0)K_R^2|x|^2\delta^{2(1-\theta)}\\
&\quad\quad~+\big(1+1/\vv_0+2/K_R^*\big)\big(b^{(\8)}_0\big)^2\delta\Big)\I_{\{x\in B_R^c\}},
\end{split}
\end{equation}
where
$
c_1:=(1+2C_R+2C_R^2+K_R^*/2)R^2+3\big(b^{(\8)}_0\big)^2.
$
Subsequently, combining    \eqref{ER0} with
\eqref{ER*} and  $|\pi^{(\delta)}(x)|\le |x|$  yields that
\begin{align*}
\Lambda(x)
&\le \big((1-K_R^*\delta/2)|x|^2+c_1\delta\big)\I_{\{x\in B_R\}}\\
&\quad+\big((1-K_R^*\delta/2)|x|^2+c_2\delta\big)\I_{\{x\in B_R^c\}},
\end{align*}
where $c_2:=
\big(1+1/\vv_0+2/K_R^*\big)\big(b^{(\8)}_0\big)^2$.
Therefore,  we have that  for any $x\in\R^d,$
\begin{align*}
\Lambda(x)
 \le (1-K_R^*\delta/2)|x|^2+(c_1\vee c_2)\delta.
\end{align*}
This, along with \eqref{EW}, gives that
\begin{align*}
\E\big(\big|X_{(n+1)\delta}^\delta\big|^2\big|X_{n\delta}^\delta\big)&= \Lambda(X_{n\delta}^\delta)+\delta\E\big(|\xi_{n+1}|^2\big|X_{n\delta}^\delta\big)\\
&\le (1-K_R^*\delta/2)|X_{n\delta}^\delta|^2+(c_1\vee c_2+d)\delta,
\end{align*}
where in the identity we  used the independence of $(\xi_n)_{n\ge1}$
 as well as $\E\xi_n={\bf0}$, while in the inequality we employed the fact that  $\E|\xi_n|^2=d$. As a consequence, \eqref{T1} follows immediately.
\end{proof}

Now, we make some comments on Lemma \ref{Lya}.
\begin{remark}
Note that the constant (i.e., $c_0\delta$) involved in  \eqref{T1} is linear with respect to the step size $\delta.$
This fact plays a crucial role in the construction of the set $\mathcal D$, introduced in \eqref{EP4}.
Suppose the constant $c_0\delta$ is replaced by a number, still written as $c_0,$ which is independent of the step size $\delta.$ In this case, by inspecting the proof of Theorem \ref{add} below,  one will take the set $\mathcal D$ that is dependent on the parameter $\delta$ (so we can write it  as $\mathcal D_\delta$). Unfortunately, the diameter of $\mathcal D_\delta$ goes to infinity as the step size $\delta$ approaches
zero, and the corresponding $\lambda_2$ in \eqref{ER6} is also dependent on $\delta$.
Unfortunately,
$\lambda_2$ tends to zero when $\delta$ goes to zero so there will be no
convergence  rate.

Concerning the tamed EM scheme for overdamped Langevin SDEs, by invoking the logarithmic Sobolev inequality
for Gaussian measures, \cite[Proposition 3]{BDMS}
examined a Lyapunov condition in the exponential form.
 In contrast,
 in Lemma \ref{Lya}, we confirm a Lyapunov condition in the polynomial form for a wide range of approximate schemes.
 Moreover, our proof is much more succinct.
\end{remark}

With the help of Lemma \ref{Lya}, we proceed to complete the
\begin{proof}[Proof of Theorem $\ref{add}$]
For $a,\vv_\star>0$   given in \eqref{T3}, define the following  distance function:
\begin{align*}
\hat\rho_2(x,y)=\big(a+\vv_\star (2+|x|^2+|y|^2 )\big)\I_{\{x\neq y\}}+h(|x-y|\wedge (1+r_{\mathcal D})),\quad x,y\in\R^d,
\end{align*}
where
$$h(r):=1-\e^{-c_\star r}, \quad r\ge0$$
with $c_\star>0$ being introduced in \eqref{T0}.  Note that the metrics $\rho_2$ and $\hat\rho_2$ are mutually equivalent.  Therefore,
in order to prove  the desired assertion \eqref{ER6},
 it suffices to show via an inductive argument that there exists $\lambda\in(0,1)$ such that
  for any $\delta\in(0,\delta_2]$, any integer $n\ge0$, and
  all $x,y\in\R^d$ with $x\neq y $,
\begin{align}\label{EP2}
\bar\E^{(x,y)}\hat\rho_{2}(X_\delta^{\delta},Y_\delta^{\delta})\le (1-\lambda n\delta)\hat\rho_2(x,y),
\end{align}
where   $(X_\delta^{\delta},Y_\delta^{\delta})$ is determined by \eqref{E4}.

In the following, we assume $x,y\in\R^d$ with $x\neq y,$ and fix $\delta\in(0,\delta_2].$
By following the reasoning of \eqref{WE}, we derive that
\begin{equation}\label{T2}
\begin{split}
 \bar\E^{(x,y)}\hat \rho_2(X_\delta^{\delta},Y_\delta^{\delta})-\hat \rho_2(x,y)
&\le -a\bar\P^{(x,y)}\big(X_\delta^{\delta}=Y_\delta^{\delta}\big) +c_\star \e^{-c_\star(|x-y|\wedge(1+r_{\mathcal D}))}\\
&\quad \times\Big(\bar\E^{(x,y)}(|X_\delta^{\delta}-Y_\delta^{\delta}|\wedge(1+r_{\mathcal D})) - (|x-y|\wedge(1+r_{\mathcal D})) \Big)\\
&\quad-\frac{1}{2}c^2_\star\e^{-c_\star(|x-y|\wedge(1+r_{\mathcal D}))}\\
&\qquad\times\bar\E^{(x,y)}\Big(\big( |X_\delta^{\delta}-Y_\delta^{\delta}|\wedge(1+r_{\mathcal D})   - (|x-y|\wedge(1+r_{\mathcal D}))\big)^2\\
&\qquad\qquad\qquad\times \I_{\{|X_\delta^{\delta}-Y_\delta^{\delta}|\wedge(1+r_{\mathcal D}) <|x-y|\wedge(1+r_{\mathcal D})\}}\Big)\\
&\quad+  \frac{\vv_\star}{2} \big( -K_R^* (2+|x|^2+|y|^2)+4c_0 \big)\delta.
\end{split}
\end{equation}

Case (i): $x,y\in\R^d$ with $0<|x-y|\le  r_0:= \delta^{{1}/{2}}/(1+C_R)$. Concerning this case, by invoking \eqref{EE} and $\kk=\delta^{{1}/{2}}$,
 we deduce from \eqref{T2}  that
\begin{equation}
\begin{split}
\bar\E^{(x,y)}\hat \rho_2(X_\delta^{\delta},Y_\delta^{\delta})-\hat \rho_2(x,y)&\le -\frac{1}{2}aJ(1) +c_\star \e^{-c_\star(|x-y|\wedge(1+r_{\mathcal D}))}\big(\bar\E^{(x,y)}|X_\delta^{\delta}-Y_\delta^{\delta}| - |x-y|) \big)\\
&\quad+  \frac{\vv_\star}{2} \big( -K_R^* (2+|x|^2+|y|^2)+4c_0  \big)\delta \\
&\le-\frac{1}{2}aJ(1)+c_\star+  \frac{\vv_\star}{2} \big( -K_R^* (2+|x|^2+|y|^2)+4c_0  \big)\delta\\
&\le -\frac{1}{4}aJ(1)- \frac{1}{2} K_R^*\vv_\star (2+|x|^2+|y|^2)\delta,
\end{split}
\end{equation}
where the first inequality is valid due to $r_0<1+r_{\mathcal D}$, and the third inequality is available by taking the
definition
of $a$ into consideration. Thus, by noting that for all $x,y\in\R^d,$
\begin{align}\label{EP1}
\hat \rho_2(x,y)\le a+1+\vv_\star (2+|x|^2+|y|^2 )
\end{align}
and that $\frac{1}{2}K_R^*\delta\le \frac{aJ(1)}{4(a+1)}$ for any $\delta\in(0,\delta_2],$
we arrive at
\begin{align*}
\bar\E^{(x,y)}\hat \rho_2(X_\delta^{\delta},Y_\delta^{\delta})\le \big(1-K_R^*\delta/2\big)\hat \rho_2(x,y),
\end{align*}
where $ 1-K_R^*\delta/2\ge0 $ for any $\delta\in(0,\delta_2].$

Case (ii): $x,y\in\R^d$ with $r_0<|x-y|\le 1+r_{\mathcal D} $.
In this case,
we obtain from
\eqref{E9*}
and
 \eqref{T2}
 that
\begin{align*}
\bar\E^{(x,y)}\hat \rho_2(X_\delta^{\delta},Y_\delta^{\delta})-\hat \rho_2(x,y)&\le c_\star \e^{-c_\star |x-y|  } C_R|x-y|
\delta
  -\frac{1}{2}c^2_\star\e^{-c_\star |x-y| }\\
&\quad\times\bar\E^{(x,y)}\Big(\big( |X_\delta^{\delta}-Y_\delta^{\delta}|   -  |x-y| )\big)^2\I_{\{|X_\delta^{\delta}-Y_\delta^{\delta}| <|x-y| \}}\Big)\\
&\quad+  \frac{\vv_\star}{2} \big( -K_R^* (2+|x|^2+|y|^2)+4c_0 \big)\delta\\
&\le  c_\star \e^{-c_\star |x-y|  } \bigg(C_R(1+r_{\mathcal D}) - \frac{J(1)c_\star}{4(1+C_R)^2}\bigg)\delta\\
&\quad+  \frac{\vv_\star}{2} \big( -K_R^* (2+|x|^2+|y|^2)+4c_0 \big)\delta\\
&\le -\frac{J(1)c_\star^2\e^{-c_\star (1+r_{\mathcal D}) }\delta }{8(1+C_R)^2}+  \frac{\vv_\star}{2} \big( -K_R^* (2+|x|^2+|y|^2)+4c_0 \big)\delta  \\
&=-2c_0c_\star\vv_\star\delta- \frac{1}{2}K_R^*\delta \vv_\star(2+|x|^2+|y|^2),
\end{align*}
where the second inequality holds true by making use of \eqref{EP}, the third inequality is verifiable  thanks to the
definition
of $c_\star$ 
in \eqref{T0},
and the identity is valid owing to the definition of $\vv_\star$  
in \eqref{T3}.
 Next, by means of the definition of $a$ given in \eqref{T3}, we  
 evidently
 have
\begin{align*}
\frac{2c_0c_\star\vv_\star}{a+1}<\frac{1}{2}K_R^*.
\end{align*}
This, together with \eqref{EP1}, leads to
\begin{align*}
\bar\E^{(x,y)}\hat \rho_2(X_\delta^{\delta},Y_\delta^{\delta})
&\le (1-\lambda_{22} \delta )\hat \rho_2(x,y),
\end{align*}
where $\lambda_{22}:=\frac{2c_0c_\star\vv_\star}{a+1}\in(0,1)$ by taking advantage of the definition of $a.$

Case (iii): $x,y\in\R^d$ with $|x-y|>1+r_{\mathcal D}$. Regarding  this setting, one trivially  has
\begin{align*}
4c_0\le \frac{1}{2}K_R^* \big(2+|x|^2+|y|^2\big).
\end{align*}
Whence, combining  \eqref{T2} with
the fact that $h$ is increasing
gives that
\begin{equation*}
\begin{split}
\bar\E^{(x,y)}\hat \rho_2(X_\delta^{\delta},Y_\delta^{\delta})-\hat \rho_2(x,y)
&\le \bar\E^{(x,y)}h(|X_\delta^{\delta}-Y_\delta^{\delta}|\wedge(1+r_{\mathcal D}))-h(1+r_{\mathcal D})\\
&\quad+  \frac{\vv_\star}{2} \big( -K_R^*  (2+|x|^2+|y|^2)+4c_0  \big)\delta\\
&\le -\frac{1}{4}K_R^* \vv_\star (2+|x|^2+|y|^2) \delta\\
&\le  -\frac{1}{4}K_R^* \vv_\star  \delta -\frac{1}{8}K_R^* \vv_\star (2+|x|^2+|y|^2) \delta.
\end{split}
\end{equation*}
The above estimate, in addition to the fact that $\vv_\star/(1+a)<1/2$ (owing to the definition of $a$) and \eqref{EP1}, ensures  that
 \begin{equation*}
\begin{split}
\bar\E^{(x,y)}\hat \rho_2(X_\delta^{\delta},Y_\delta^{\delta})
&\le   (1-\lambda_{23}  )\hat \rho_2(x,y),
\end{split}
\end{equation*}
in which $\lambda_{23}:=\frac{K_R^* \vv_\star }{4(a+1)}\in(0,1)$ by taking the definition of $a$ into account.

Finally, the assertion \eqref{EP2} with $\lambda:=\lambda_{21}\wedge\lambda_{22}\wedge\lambda_{23}$, where $\lambda_{21}:=\frac{1}{2}K_R^*,$ follows by summarizing the analysis above.
\end{proof}

\section{Proof  of Theorem \ref{W1}}\label{sec4}
The proof of Theorem \ref{W1}  is also based on the coupling approach. For this purpose, we construct the following iteration procedure: for any $\delta,\kk^*>0$ and integer $n\ge0,$
 \begin{equation}\label{ER4}
\begin{cases}
 X_{(n+1)\delta}^{\delta}=\hat {X}_{n\delta}^{\delta}+\delta^{1/2}\xi_{n+1}\\
 Y_{(n+1)\delta}^{\delta}=\hat {Y}_{n\delta}^{\delta}+\delta^{1/2}\Big\{\big(\xi_{n+1}+\delta^{-{1}/{2}}(\hat {Z}_{n\delta}^{\delta})_{\kk^*}\big)\I_{\{U_{n+1}\le  \rho(-\delta^{-{1}/{2}} (\hat {Z}_{n\delta}^{\delta})_{\kk^*},\xi_{n+1})\}}\\
 \qquad\qquad\qquad\qquad \quad+\Pi(\hat {Z}_{n\delta}^{\delta})\xi_{n+1}\I_{\{  (\rho(-\delta^{-{1}/{2}}(\hat {Z}_{n\delta}^{\delta})_{ \kk^*},\xi_{n+1})\le U_{n+1}\le 1\}} \Big\}.
\end{cases}
\end{equation}
Herein,  $\hat {Z}_{n\delta}^{\delta}:=\hat {X}_{n\delta}^{\delta}-\hat {Y}_{n\delta}^{\delta}$ with  $\hat {X}_{n\delta}^{\delta}$ and $\hat {Y}_{n\delta}^{\delta}$ being defined as in \eqref{EE8}, and
for any $x\in\R^d,$
\begin{align}\label{EWW}
\Pi(x):=I_d-\frac{2x   x^\intercal}{| x|^2}\I_{\{| x|\neq0\}}\in\R^d\otimes\R^d,
\end{align}
where  $I_d$ stands for the $d\times d$ identity matrix, and
$x^\intercal$ means  the transpose of $x\in\R^d.$
Note that some variants of this coupling have been used to study contractivity of classical EM schemes corresponding to SDEs with coefficients with linear growth; see  e.g.\ in \cite{EM, MMS}.

The lemma below  provides a streamlined proof
that $(X_{n\delta}^\delta,Y_{n\delta}^\delta)_{\ge0}$ defined in \eqref{ER4}
is a coupling process of $(X_{n\delta}^\delta)_{n\ge0} $
defined in  \eqref{EW-}.

\begin{lemma}\label{cou}
For any $\delta,\kk^*>0,$ $(X_{n\delta}^\delta,Y_{n\delta}^\delta)_{\ge0}$ determined by  \eqref{ER4}
is a coupling process of $(X_{n\delta}^\delta)_{n\ge0}.$
\end{lemma}

\begin{proof}
As analyzed at the beginning of the proof of Lemma \ref{coupling}, in order to prove the desired assertion, it is sufficient to  show that,
 for $(X_{n\delta}^\delta,Y_{n\delta}^\delta)_{\ge0}$ solving \eqref{ER4}, equation \eqref{E5} holds.

It is easy to see that
\begin{align*}
\bar \E^{(x,y)}f(Y_\delta^\delta)&=\int_{\R^d}f\big(\hat y^\delta+\delta^{1/2} (u+\delta^{-{1}/{2}} (\hat z^\delta)_{\kk^*} )\big)\,\nu_{-\delta^{-{1}/{2}}(\hat z^{\delta})_{\kk^*}}(\d u)\\
&\quad+\int_{\R^d}f\big(\hat y^\delta+  \delta^{1/2} \Pi(\hat z^\delta) u\big)\Big(\nu(\d u)-\nu_{-\delta^{-{1}/{2}} (\hat z^{\delta})_{\kk^*}}(\d u)\Big).
\end{align*}
By changing the variable $u\rightarrow u-\delta^{-{1}/{2}}(\hat z^\delta)_{\kk^*} $,  and using
 $\nu_{-x}(\d (u-x))=\nu_x(\d u)$ for $x\in\R^d$,
 along with the rotational invariance of $\nu(\d z)$ and the fact that the matrix $\Pi$ defined in \eqref{EWW} is orthogonal,
  we find that
\begin{align*}
\bar \E^{(x,y)}f(Y_\delta^\delta)&=\int_{\R^d}f(\hat y^\delta+\delta^{1/2} u ) \,\nu_{\delta^{-{1}/{2}}(\hat z^{\delta})_{\kk^*}}(\d u)+\int_{\R^d}f(\hat y^\delta+  \delta^{1/2} u)\nu(\d u)\\
&\quad-\int_{\R^d}f(\hat y^\delta+  \delta^{1/2}\Pi(\hat z^\delta) u) \nu_{-\delta^{-{1}/{2}}(\hat z^{\delta})_{\kk^*}}(\d u).
\end{align*}
Whence, \eqref{E5} is available as soon as
\begin{align}\label{EE1}
\int_{\R^d}f(\hat y^\delta+\delta^{1/2} u )
\,\nu_{\delta^{-{1}/{2}}(\hat z^{\delta})_{\kk^*}}(\d u)=\int_{\R^d}f(\hat y^\delta+  \delta^{1/2}\Pi(\hat z^\delta) u) \nu_{-\delta^{-{1}/{2}}(\hat z^{\delta})_{\kk^*}}(\d u).
\end{align}
Indeed,  since $\Pi(\cdot)$ is an orthogonal matrix and   $\Pi(x) x=- x$ for any $x\in\R^d,$ we have
\begin{align*}
|u+\delta^{1/2}(\hat z^{\delta})_{\kk^*}|=|\Pi(\hat z^\delta)u+\delta^{1/2}\Pi(\hat z^\delta)(\hat z^{\delta})_{\kk^*}|=|\Pi(\hat z^\delta)u-\delta^{1/2}(\hat z^{\delta})_{\kk^*}|.
\end{align*}
Thus, via changing
the variable $u\rightarrow\Pi(\hat z^\delta) u,$  \eqref{EE1} follows immediately.
\end{proof}

In the sequel, we shall fix $\kk^*=r_0\delta^{1/2}/2$ for any $\delta\in(0,\delta_3],$ where $r_0$ and $\delta_3$ were defined in \eqref{EW9} and \eqref{EE0}, respectively.

 \begin{lemma}\label{lem2}
Under $({\bf H}_1')$ and $({\bf H}_2')$, for any $\delta\in(0,\delta_3]$ and
 $x,y\in\R^d $ with $|x-y|\le 2R,$
 \begin{equation}\label{LL}
\begin{split}
\Phi(x,y):&=\bar\E^{(x,y)}\Big(\big(|X_\delta^\delta-Y_\delta^\delta|-|x-y|\big)^2\I_{\{|X_\delta^\delta-Y_\delta^\delta|-|x-y|\le \gamma\delta^{1/2}\}}\Big) \ge r^*_0 \delta^{1/2} (|\hat z^\delta|\wedge\kk^*),
\end{split}
\end{equation}
where $\hat z^{\delta}:=\hat x^{\delta}-\hat y^{\delta}$ with $\hat x^\delta := x +b^{(\delta)}( x )\delta$, and
$r_0^*,\gamma>0$   were defined
in \eqref{EW9}.
 \end{lemma}

\begin{proof}
For a related argument in the context of classical EM schemes, see \cite[Lemma 2.7]{EM} or \cite[Lemma 3.2]{MMS}.
In the following analysis, we stipulate $x,y\in\R$ with $|x-y|\le 2R $ and fix $\delta\in(0,\delta_3]$ so that
\begin{align}\label{EE3}
R(C_R\vee K_R)\delta^{{1}/{2}-\theta}\le  1/2.
\end{align}
Obviously, \eqref{LL} holds true in case of $|\hat z^\delta|=0.$ So, it suffices  to verify \eqref{LL} for the case  $|\hat z^\delta|>0$.
Note obviously from \eqref{ER4} that
\begin{align*}
X_\delta^\delta-Y_\delta^\delta=\hat z^\delta-(\hat z^\delta)_{\kk^*}\I_{\{U_1\le  \rho(-\delta^{-{1}/{2}} (\hat z)_{\kk^*},\xi_1)\}}+\frac{2\delta^{1/2}\hat z^\delta}{|\hat z^\delta|^2}\<\hat z^\delta,\xi_1\> \I_{\{  (\rho(-\delta^{-{1}/{2}}(\hat z^{\delta})_{\kk^*},\xi_1)\le U_1\le 1\}}.
\end{align*}
Whence, we find that
\begin{align*}
\Phi(x,y)&\ge \int_{\R^d}\big(\big||\hat z^\delta|+2\delta^{1/2}\<\hat z^\delta/|\hat z^\delta|,u\>\big| -|x-y|\big)^2\I_{\{||\hat z^\delta|+2\delta^{1/2}\<\hat z^\delta/|\hat z^\delta|,u\>|-|x-y|\le\gamma\delta^{1/2}\}}\\
&\quad\quad\times\Big(\nu(\d u)-\nu_{-\delta^{-{1}/{2}} (\hat z^{\delta})_{\kk^*}}(\d u)\Big).
\end{align*}
Let $\bar\Pi\in\R^d\otimes\R^d$ be an orthogonal matrix such that $\bar \Pi \hat z^\delta=|\hat z^\delta|(1,0,\cdots,0)\in\R^d.$ Then, applying Jacobi's transformation rule, in addition to $|\bar\Pi u|=|u|$ for $u\in\R^d,$ yields that
\begin{align*}
\Phi(x,y)&\ge \int_{\R^d}\big(\big||\hat z^\delta|+2\delta^{1/2}\<\hat z^\delta/|\hat z^\delta|,\bar \Pi u\>\big| -|x-y|\big)^2\I_{\{||\hat z^\delta|+2\delta^{1/2}\<\hat z^\delta/|\hat z^\delta|,\bar \Pi u\>|-|x-y|\le\gamma\delta^{1/2}\}}\\
&\quad\quad\times\frac{1}{(2\pi)^{\frac{d}{2}}}\Big(\e^{-\frac{|  u|^2}{2}}-\e^{-\frac{|  u|^2}{2}}\wedge\e^{-\frac{1}{2}(|u|^2+2\delta^{-{1}/{2}}\<\bar
\Pi u,(\hat
z^{\delta})_{\kk^*}\>   +\delta^{-1} |(\hat z^{\delta})_{\kk^*}|^2)} \Big)\,\d u\\
&=\int_{\R^d}\big(\big||\hat z^\delta|+2\delta^{1/2}u_1\big| -|x-y|\big)^2\I_{\{||\hat z^\delta|+2\delta^{1/2}u_1|-|x-y|\le\gamma\delta^{1/2}\}}\\
&\quad\quad\times\frac{1}{(2\pi)^{\frac{d}{2}}}\Big(\e^{-\frac{|  u|^2}{2}}-\e^{-\frac{|  u|^2}{2}}\wedge\e^{-\frac{1}{2}(|u|^2+2\delta^{-{1}/{2}}(|\hat z^\delta|\wedge{\kk^*})u_1   +\delta^{-1} (|\hat z^\delta|\wedge{\kk^*})^2)} \Big)\,\d u\\
&\ge\int_{\R}\big(\big||\hat z^\delta|+2\delta^{1/2}u\big| -|x-y|\big)^2\I_{\{\delta^{1/2}\le ||\hat z^\delta|+2\delta^{1/2}u|-|x-y|\le\gamma\delta^{1/2}\}}\\
&\quad\quad\times\frac{1}{(2\pi)^{\frac{1}{2}}}\Big(\e^{-\frac{  u^2}{2}}-\e^{-\frac{u^2}{2}}\wedge\e^{-\frac{1}{2}((u+\delta^{-{1}/{2}}(|\hat z^\delta|\wedge{\kk^*}))^2} \Big)\,\d u,
\end{align*}
where $u_1$ is the first component of the vector $u.$
It is easy to see that
\begin{align*}
\Phi(x,y)
&\ge \frac{\delta}{(2\pi)^{\frac{1}{2}}}\bigg(\int_{-(2\delta^{1/2})^{-1}(|\hat z^\delta|\wedge\kk^*)}^\8\I_{\{\Gamma(x,y)+\frac{1}{2}\le  u\le \Gamma(x,y)+\frac{\gamma}{2} \}} \Big(\e^{-\frac{ u^2}{2}}- \e^{-\frac{1}{2}((u+\delta^{-{1}/{2}}(|\hat z^\delta|\wedge\kk^*))^2} \Big)\,\d u,
\end{align*}
where
$
\Gamma(x,y):=(2\delta^{ \frac{1}{2}})^{-1}(|x-y|-|\hat z^\delta|).$ Next, by the triangle inequality, in addition to the definition of $\hat z^\delta$,
we infer from \eqref{EW4-} that
\begin{align}\label{ET1}
 ||\hat z^\delta|-|x-y||\le |\hat z^\delta-(x-y)|\le (C_R\vee K_R)\delta^{1-\theta}|x-y|\le 2R (C_R\vee K_R)\delta^{1-\theta}.
\end{align}
This, together with \eqref{EE3}, yields that
\begin{align*}
0\le\Gamma(x,y)+1/2
\le r_0 \quad\mbox{ and } \quad   \Gamma(x,y)+\gamma/2 \ge2r_0.
\end{align*}
Thus, by applying  Fubini's theorem and taking advantage of $\delta^{-{1}/{2}}\kk^*=r_0/2$ owing to the choice of $\kk^*,$
it follows that
\begin{align*}
\Phi(x,y)
&\ge\frac{\delta}{(2\pi)^{\frac{1}{2}}}\int_{r_0}^{2r_0}\big(\e^{-\frac{1}{2}u^2}-\e^{-\frac{1}{2}(u+\delta^{-{1}/{2}}(|\hat z^\delta|\wedge\kk^*))^2}\big)\,\d u\\
&\ge-\frac{\delta}{(2\pi)^{\frac{1}{2}}}\int_{\frac{3}{2}r_0}^{2r_0}\int_{s-\delta^{-{1}/{2}}(|\hat z^\delta|\wedge\kk^*)}^s\big(\e^{-\frac{1}{2}|s|^2}\big)' \,\d u\,\d s\\
&=r_0^* \delta^{1/2} (|\hat z^\delta|\wedge\kk^*).
\end{align*}
Consequently, \eqref{LL} is available.
\end{proof}

Recall that $c^*>0$ and $\delta_3>0$ were defined in  \eqref{ET6} and \eqref{EE0}, respectively.   We readily have for any    $\delta\in(0, \delta_3]$,
\begin{align}\label{EW10}
2C_R<\frac{c^*r^*_0 r_0}{8R},\quad \e^{-c^* \gamma\delta^{1/2} }\ge\frac{1}{2},\quad C_R\e^{-2Rc^* }<1,\quad K_R^*\e^{- 2Rc^* }<1,
\end{align}
and
\begin{align}\label{ET2}
 (C_R\vee K_R)\delta^{1-\theta}\le \frac{1}{2},\quad
 \frac{r_0\delta^{{1}/{2}}}{R}\le1,
  \quad K_R^2\delta^{2(1- \theta)}\le K_R^*\delta,\quad 1-2K_R^*\delta\ge0.
\end{align}
For any $\delta\in(0, \delta_3],$
define the function
\begin{equation*}
\varphi(r)=
\begin{cases}
\frac{1}{c^*}(1-\e^{-c^*r}),\qquad\qquad\qquad\qquad\qquad\qquad \quad\quad\quad\qquad 0\le r\le 2R+\gamma\delta^{1/2},\\
\frac{1}{c^*}(1-\e^{-c^*(2R+\gamma\delta^{1/2})})+\frac{1}{2}\e^{-c^*(2R+\gamma\delta^{1/2})}\big(r-(2R+\gamma\delta^{1/2})\big)\\
+\frac{1}{4c^*}\e^{-c^*(2R+\gamma\delta^{1/2})}\big(1-\e^{-2c^*(r-(2R+\gamma\delta^{1/2}))}\big),\qquad \qquad\qquad r>2R+\gamma\delta^{1/2},
\end{cases}
\end{equation*}
where $\gamma>0$ was introduced  in \eqref{EW9}.

With the help of Lemmas \ref{cou} and \ref{lem2}, we proceed to implement the
\begin{proof}[Proof of Theorem $\ref{W1}$]
The proof of Theorem \ref{W1} is essentially inspired by that of \cite[Theorem 2.5]{HMW}. By following the strategy to complete the proof of Theorem \ref{thm} and noting that $c_0r\le \varphi(r)\le r$ for
$c_0: = \frac{1}{2}e^{c^*(2R + \gamma \delta^{1/2})}$,
it suffices to verify that  for any $\delta\in(0,\delta_3]$ and $x,y\in\R^d $ with $x\neq y,$
\begin{align}\label{EW13}
\bar\E^{(x,y)}\varphi\big( |X_\delta^{\delta}-Y_\delta^{\delta}|\big)\le (1-\lambda_3\delta)\varphi( |x-y|),
\end{align}
where $\lambda_3:=\lambda_{31}\wedge\lambda_{32}$ with $\lambda_{31},\lambda_{32}>0$ being given in \eqref{EE7}.
Then \eqref{W1_contractivity} follows with $C_3^* = 1/c_0$.

It is easy to see that $\varphi\in C^2([0,\8))$ such that for $r\le 2R+\gamma\delta^{1/2}$
\begin{align}\label{EW0}
\varphi'(r)=\e^{-c^*r}>0,\quad \varphi''(r)=-c^*\e^{-c^*r}<0,
\end{align}
and for $r>2R+\gamma\delta^{1/2}$,
\begin{align}\label{E-}
\varphi'(r)&=\frac{1}{2}\e^{-c^*(2R+\gamma\delta^{1/2})}\Big(1+\e^{-2c^*(r-(2R+\gamma\delta^{1/2}))}\Big)>0,~ \varphi''(r)=-c^*   \e^{ c^*(-2r+ 2R+\gamma\delta^{1/2} )}<0.
\end{align}
Then, by the Taylor expansion, we derive from $\varphi''<0$ and $\varphi'''>0$ that
\begin{equation}\label{ET8}
\begin{split}
 \bar\E^{(x,y)}\varphi\big( |X_\delta^{\delta}-Y_\delta^{\delta}|\big)-\varphi(|x-y|)
 &\le \varphi'(|x-y|)\big(\bar\E^{(x,y)}|X_\delta^{\delta}-Y_\delta^{\delta}|-|x-y|\big)\\
 &\quad +\varphi''(|x-y|+\gamma\delta^{1/2})\Phi(x,y),
\end{split}
\end{equation}
where $\Phi$ was defined as in \eqref{LL}.
By mimicking the procedure to derive  \cite[Lemma 3.8]{HMW}, we have
\begin{align*}
\bar\E^{(x,y)}|X_\delta^{\delta}-Y_\delta^{\delta}|=|\hat z^\delta|.
\end{align*}
Using \eqref{ET2} and \eqref{EW4-} similarly as in the proof of Lemma \ref{lem1}, we derive
\begin{align}\label{ET9}
\bar\E^{(x,y)}|X_\delta^{\delta}-Y_\delta^{\delta}|\le\Big(1+C_R\delta  \I_{\{|x-y|\le 2R\}}- \frac{1}{2}K_R^*\delta \I_{\{|x-y|>2R\}}\Big)|x-y|.
\end{align}
Subsequently,
\eqref{ET9}
along with
\eqref{EW10},
  as well as \eqref{EW0},
 enables us to derive  that for any $x,y\in\R^d$ with $0<|x-y|\le 2R$,
\begin{equation}\label{ET4}
\begin{split}
 \bar\E^{(x,y)}\varphi\big( |X_\delta^{\delta}-Y_\delta^{\delta}|\big)-\varphi(|x-y|)
 &\le\e^{-c^*|x-y|}\big(C_R\delta|x-y|-c^*\e^{-c^* \gamma\delta^{1/2} }\Phi(x,y)\big)\\
 &\le \e^{-c^*
 |x-y|}\bigg(C_R\delta-\frac{c^* \Phi(x,y)}{2|x-y|}\bigg)|x-y|.
\end{split}
\end{equation}

By using  \eqref{ET1}, we obviously have
\begin{align*}
 |\hat z^\delta| \ge|x-y|-(C_R\vee K_R)\delta^{1-\theta}|x-y|.
\end{align*}
 Consequently,  \eqref{LL}  and
 \eqref{ET2} imply  that  for any $x,y\in\R^d$ with $0<|x-y|\le 2R$,
\begin{align*}\label{ET3}
\frac{ \Phi(x,y)}{|x-y|}  \ge r^*_0\delta^{1/2}\bigg(\frac{1}{2}\wedge\frac{\kk^*}{|x-y|}\bigg)\ge\frac{1}{2} r^*_0\delta^{1/2}\bigg(1\wedge\frac{\kk^*}{ R}\bigg)=\frac{ r^*_0   r_0\delta}{4R},
\end{align*}
where $\kk^*=r_0\delta^{1/2}/2$.
Plugging this estimate  back into \eqref{ET4} and taking $\varphi(r)\le r$ for all $r\ge0,$ and \eqref{EW10}  into account
 gives that    for any $x,y\in\R^d$ with $0<|x-y|\le 2R$,
 \begin{equation}\label{EW11}
 \bar\E^{(x,y)}\varphi\big( |X_\delta^{\delta}-Y_\delta^{\delta}|\big) \le \varphi(|x-y|) -C_R\delta\e^{-2Rc^* }|x-y| \le \big(1-\lambda_{31}\delta\big)\varphi(|x-y|),
\end{equation}
where $\lambda_{31}:=C_R\e^{-2Rc^* }\in(0,1)$
 due to \eqref{EW10}.

By making use of \eqref{ET8} and \eqref{ET9}, in addition to $\Phi\ge0 $ and  $\varphi''<0$,   it follows from \eqref{EW0} and \eqref{E-} that for any $x,y\in\R^d$ with $|x-y|\ge2R,$
\begin{equation}\label{EW12}
\begin{split}
 \bar\E^{(x,y)}\varphi\big( |X_\delta^{\delta}-Y_\delta^{\delta}|\big)
 &\le \varphi(|x-y|) - \frac{1}{4}K_R^*\e^{-c^*(2R+\gamma)}\delta|x-y|\\
 &\le \big(1- \lambda_{32}\delta\big) \varphi(|x-y|),
\end{split}
\end{equation}
where the second inequality is valid by noticing that for $r\ge 2R$ and $R\ge1,$
\begin{align*}
\frac{r}{\varphi(r)}\ge \frac{r}{1-R+  r/2}\ge 2,
\end{align*}
 and $\lambda_{32}:=\frac{1}{2} K_R^*\e^{-c^*(2R+\gamma)}\in(0,1)$
  thanks to \eqref{EW10}.

At last, the desired assertion \eqref{EW13} follows by combining \eqref{EW11} with \eqref{EW12}.
\end{proof}

\section{Proof of Theorem \ref{IPM}}\label{sec5}
We adopt the following proof strategy. First of all, we will decompose the Wasserstein distance $\mathbb W_1\big(\pi,\pi^{\delta}\big)$ by using the triangle inequality to estimate
\begin{equation}\label{main_estimate_section6}
\begin{split}
\mathbb W_1\big(\pi,\pi^{\delta}\big)&=\mathbb W_1\left(\mathscr L_{X_{n\delta}^\pi},\mathscr L_{X_{n\delta}^{\delta,\pi^{\delta}}}\right)\\
&\le \mathbb W_1\Big(\mathscr L_{X_{n\delta}^\pi},\mathscr L_{X_{n\delta}^{{\bf0}}}\Big)+\mathbb W_1\Big(\mathscr L_{X_{n\delta}^{{\bf0}}},\mathscr L_{X_{n\delta}^{\delta,{\bf0}}}\Big)+\mathbb W_1\Big(\mathscr L_{X_{n\delta}^{\delta,{\bf0}}},\mathscr L_{X_{n\delta}^{\delta,\pi^\delta}}\Big)\\
&=:\varphi_1(n,\delta)+\varphi_2(n,\delta)+\varphi_3(n,\delta).
\end{split}
\end{equation}
Here $(X_{n\delta}^{{\bf0}})_{n\ge0}$ stands for the solution to \eqref{E1} with the initial value $X_0={\bf 0}$, and $(X_{n\delta}^{\delta,{\bf0}})_{n\ge0}$ denotes the tamed scheme \eqref{ER1} with the initial value $X_0^\delta={\bf 0}$, whereas $\pi$ and $\pi^{\delta}$ are their respective invariant probability measures.
The term $\varphi_1(n,\delta)$ can be bounded by employing a result from the literature on exponential ergodicity of SDEs with one-sided Lipschitz drifts that are dissipative at infinity; see, for example,    \cite[Corollary 2]{Eberle}. A bound on $\varphi_3(n,\delta)$ follows from our Theorem \ref{W1}. We will therefore focus on term $\varphi_2(n,\delta)$, which we will bound by utilizing a coupling between two diffusions with different drifts (Proposition \ref{pro}) and a uniform bound on the moments of the tamed Euler schemes (Lemma \ref{lem5.2}).

To begin our work on bounding $\varphi_2(n,\delta)$, we consider the following coupling, for any $\vv,t>0$,
\begin{equation}\label{EE17}
\begin{cases}
\d \bar X_t=b(\bar X_t)\,\d t+ h_\vv(|\bar Z_t|)\d \bar W_t+\big(1-h_\vv(|\bar Z_t|)^2\big)^{{1}/{2}}\,\d \bar B_t,\quad\quad\quad\quad\quad\quad \bar X_0={\bf0},\\
\d \bar Y_t=b^{(\delta)}(\bar Y_{\lfloor t/\delta\rfloor\delta})\,\d t+\Pi( \bar Z_t )h_\vv(|\bar Z_t|)\d \bar W_t+\big(1-h_\vv(|\bar Z_t|)^2\big)^{{1}/{2}}\,\d \bar B_t,\quad \bar Y_0={\bf0},
\end{cases}
\end{equation}
where $\bar Z_t:=\bar X_t-\bar Y_t$;
$(\bar W_t)_{t\ge0}$ and $(\bar B_t)_{t\ge0}$ are mutually independent Brownian motions defined on the same probability space as
$(W_t)_{t\ge0}$; $\Pi(\cdot)$ is the orthogonal matrix defined  in \eqref{EWW};
and for each $\vv>0,$ $h_\vv:[0,\8)\to[0,1]$ is a continuous function  defined by
\begin{equation*}
h_\vv(r)=
\begin{cases}
0,\qquad\qquad \quad0\le r\le \vv,\\
1-\e^{-\frac{r-\vv}{2\vv-r}},\quad r\in(\vv,2\vv),\\
1, \qquad\qquad \quad r\ge2\vv.
\end{cases}
\end{equation*}

We will now show that
the $L^1$-Wasserstein distance between $(\mathscr L_{X_{n\delta}^{{\bf0}}})_{n\ge0}$ and $(\mathscr L_{X_{n\delta}^{\delta,{\bf0}}})_{n\ge0}$
can be dominated by the error between $b$ and $b^{(\delta)}$.
Note that the method of using a combination of the reflection and the synchronous coupling similar to \eqref{EE17} has been applied in the literature in many settings; see, among others, \cite[Section 6]{Eberle} for applications to interacting diffusions, \cite[Section 5]{Majka} for bounds on Malliavin derivatives or \cite{Zimmer} for the study of infinite dimensional diffusions. In the present paper we apply this technique for the first time to study tamed EM schemes.

\begin{proposition}\label{pro}
Under \eqref{EE9} and \eqref{EE11}, there exist constants $C_*,\lambda_*>0$ such that for any integer  $n\ge0,$
\begin{align}\label{EE15}
\mathbb W_1\Big(\mathscr L_{X_{n\delta}^{{\bf0}}},\mathscr L_{X_{n\delta}^{\delta,{\bf0}}}\Big)\le
C_*\int_0^{n\delta}\e^{-\lambda_*(n\delta-s)} \E\big|b(\bar Y_s)-b^{(\delta)}(\bar Y_{\lfloor s/\delta\rfloor\delta})\big|\,\d s.
\end{align}

\end{proposition}

\begin{proof}
By noting that $(\bar W_t)_{t\ge0}$ and $(\bar B_t)_{t\ge0}$ are mutually independent and that $\Pi(\cdot)$ is an orthogonal matrix, via L\'{e}vy's characterization for Brownian motions,
we conclude that for any integer $n\ge0,$
 $$\mathscr L_{\bar X_{n\delta}}=\mathscr L_{X_{n\delta}^{\bf0}}\quad \mbox{ and } \quad \mathscr L_{\bar Y_{n\delta}}=\mathscr L_{X_{n\delta}^{\delta,{\bf0}}}.$$
Therefore, \eqref{EE15} is available
provided that there exist constants $C_{*},\lambda_{*}>0$ such that for any integer $n\ge0,$
\begin{align}\label{EE5}
\bar\E|\bar X_{n\delta}-\bar Y_{n\delta}|\le C_{*}\int_0^{n\delta}\e^{-\lambda_{*}(n\delta-s)} \E\big|b(\bar Y_s)-b^{(\delta)}(\bar Y_{\lfloor s/\delta\rfloor\delta})\big|\,\d s,
\end{align}
where $\bar \E$ is the expectation operator under the probability measure $\bar\P$, the law of $(\bar X_{n\delta},\bar Y_{n\delta})_{n\ge0}$.

For $V_a(x):=\big(a+|x|^2\big)^{1/2}$ with $a>0,$
it is  immediate to see that for all $x\in\R^d,$
\begin{align*}
\nn V_a(x)=\frac{x}{V_a(x)}  \quad  \mbox{ and } \quad
\nn^2 V_a(x)=\frac{1}{V_a(x)}I_d-\frac{x x^\intercal}{V_a(x)^3}.
\end{align*}
Then,  It\^o's formula gives that
\begin{align*}
\d  V_a(\bar Z_t)&=\bigg(\frac{1}{V_a(\bar Z_t)}\<\bar Z_t,b(\bar X_t)-b^{(\delta)}(\bar Y_{\lfloor t/\delta\rfloor\delta})\>+\frac{2ah_\vv(|\bar Z_t|)^2}{V_a(\bar Z_t)^3}\bigg)\,\d t+\frac{2h_\vv(\bar Z_t)}{V_a(\bar Z_t)}\<\bar Z_t,\d \bar W_t\>.
\end{align*}
By taking the definition of $h_\vv$ into consideration, we obviously have
\begin{align*}
\frac{x}{V_a(x)}\overset{a\to0}{\longrightarrow}\frac{x}{|x|}\I_{\{|x|\neq0\}} \quad\mbox{ and } \quad  \frac{ah_\vv(|x|)}{V_a(x)^3}\le \frac{a}{(a+\vv^2)^{\frac{3}{2}}}\overset{a\to0}{\longrightarrow}0.
\end{align*}
 Subsequently, we derive that
\begin{align*}
\d  |\bar Z_t|
&\le \I_{\{|\bar Z_t|\neq0\}}\big( \<{\bf n}(\bar Z_t),b(\bar X_t)-b(\bar Y_t)\>\,\d t+ 2h_\vv(|\bar Z_t|) \<{\bf n}(\bar Z_t),\d \bar W_t\>\big)\\
&\quad+\big|b(\bar Y_t)-b^{(\delta)}(\bar Y_{\lfloor t/\delta\rfloor\delta})\big|\,\d t,
\end{align*}
in which ${\bf n}(x):=x/|x|$ for $x\neq {\bf0}$.

In order to obtain \eqref{EE5}, the key ingredient is to modify the metric
induced by $|\cdot|$.  To this end, we introduce the following $C^2$-function:
\begin{align*}
\phi(r)=c_1r+1-\e^{-c_2r},\quad r\ge0,
\end{align*}
where for $\ell_0:=2R,$
\begin{align*}
c_1:=c_2\e^{-c_2\ell_0} \quad \mbox{ and } \quad  c_2=2\lambda_1\ell_0.
\end{align*}
A direct calculation shows that for any $r\ge0,$
\begin{align*}
 \phi'(r)=c_1+c_2\e^{-c_2r} \quad \mbox{ and } \quad \phi''(r)=-c_2^2\e^{-c_2r}.
\end{align*}
Next, by virtue of \eqref{EE9} and \eqref{EE11}, it follows that for any $x,y\in\R^d,$
\begin{equation}\label{EE14}
\begin{split}
\<x-y,b(x)-b(y)\>&\le|x-y|\cdot|b(x)-b(y)| \I_{\{x\in B_R\}\cap\{y\in B_R\}}\\
&\quad+ \<x-y,b(x)-b(y)\>\I_{\{x\in B_R^c\}\cup\{y\in B_R^c\}} \\
&\le \lambda_1|x-y|^2\I_{\{|x-y|\le 2R\}}-\lambda_2|x-y|^2\I_{\{|x-y|>
2R\}},
\end{split}
\end{equation}
where $\lambda_1:=L_1(1+2\varphi(R))$ and $\lambda_2:=L_3(1+\varphi(R))$.
Thus, applying It\^o's formula and making use of \eqref{EE14} and $ \phi'\le 2c_2$
yields that
\begin{equation}\label{EE2}
\begin{split}
\d \phi(|\bar Z_t|)&\le \I_{\{|\bar Z_t|\neq0\}}\big((c_1+c_2\e^{-c_2|\bar Z_t|})\<{\bf n}(\bar Z_t),b(\bar X_t)-b(\bar Y_t)\>-2c_2^2\e^{-c_2|\bar Z_t|}h_\vv(|\bar Z_t|)^2\big)\,\d t\\
&\quad+ 2c_2\big|b(\bar Y_t)-b^{(\delta)}(\bar Y_{\lfloor t/\delta\rfloor\delta})\big|\,\d t+\d M_t\\
&\le \Lambda_\vv(|\bar Z_t|) \,\d t+ 2c_2\big|b(\bar Y_t)-b^{(\delta)}(\bar Y_{\lfloor t/\delta\rfloor\delta})\big|\,\d t+\d M_t,
\end{split}
\end{equation}
where $(M_t)_{t\ge0}$ is a martingale and
\begin{align*}
\Lambda_\vv(r):=\big((c_1+c_2\e^{-c_2r})(\lambda_1\I_{\{r\le\ell_0\}}-\lambda_2\I_{\{r>\ell_0\}})r-2c_2^2\e^{-c_2r}h_\vv(r)^2\big),\quad r\ge0.
\end{align*}
Notice from the   definitions of $c_1$ and $c_2$ that
\begin{align*}
\Lambda_\vv(r)&\le 2c_2\e^{-c_2r} \big( \lambda_1 r-c_2  h_\vv(r)^2\big)\I_{\{r\le\ell_0\}} - c_1 \lambda_2r\I_{\{r>\ell_0\}}\\
&= 2c_2\e^{-c_2r} \big(\lambda_1 r -c_2 \big)h_\vv(r)^2\I_{\{r\le\ell_0\}} - c_1 \lambda_2r\I_{\{r>\ell_0\}}+2\lambda_1 c_2\e^{-c_2r} r(1-h_\vv(r)^2)\I_{\{r\le\ell_0\}} \\
&\le - c_2^2\e^{-c_2\ell_0} h_\vv(r)^2\I_{\{r\le\ell_0\}}- c_1 \lambda_2r\I_{\{r>\ell_0\}}+4\lambda_1 c_2 r (1-h_\vv(r)),
\end{align*}
where in the last display we used $h_\vv\in[0,1].$ The previous estimate,
in addition to the fact that
\begin{align*}
\phi(r)\le 1+c_1\ell_0,\quad 0\le r\le\ell_0; \quad \frac{r}{\phi(r)}\ge\frac{\ell_0}{c_1\ell_0+1-\e^{-c_2\ell_0}}>\frac{\ell_0}{c_1\ell_0+1},\quad r\ge\ell_0,
\end{align*}
implies that for some constant $c_3>0$
\begin{align*}
\Lambda_\vv(r)
&\le -\frac{c_1c_2 }{ 1+c_1\ell_0 }\phi(r) \I_{\{r\le\ell_0\}}-\frac{ c_1 \lambda_2\ell_0}{1+c_1\ell_0} \phi(r) \I_{\{r>\ell_0\}}+ c_3(\phi(r)+r)(1-h_\vv(r) ) \\
&\le -\lambda^\star \phi(r)+c_3(2\vv+\phi(2\vv)),
\end{align*}
in which $\lambda^\star:= c_1\ell_0( (2\lambda_1) \wedge\lambda_2)/(1+c_1\ell_0).$ As a consequence, by approaching $\vv\to 0$, we deduce from \eqref{EE2}, along with $\bar X_0=\bar Y_0$,
that
\begin{align*}
\bar \E \phi(|\bar Z_t|)\le 2c_2\int_0^t\e^{-\lambda^\star(t-s)} \E\big|b(\bar Y_s)-b^{(\delta)}(\bar Y_{\lfloor s/\delta\rfloor\delta})\big|\,\d s.
\end{align*}
Whence, \eqref{EE5} follows by noting that there exist constants $c_4,c_5>0$ such that $c_4r\le\phi(r)\le c_5r$ for all $r\ge0.$
\end{proof}

In the sequel,  we  claim that $(\bar Y_{n\delta}^{{\bf0}})_{n\ge0}$, determined in \eqref{EE17}, has
moments bounded uniformly in time.

\begin{lemma}\label{lem5.2}
Under \eqref{EE9} and \eqref{EE11}, for any $p>0$, there exists a constant $C_p>0$ such that for any  $\delta\in(0,\delta_4] $ and integer $n\ge0$,
\begin{align}\label{EE18}
 \E\big|\bar Y_{n\delta}^{{\bf0}}\big|^p\le C_p.
\end{align}
\end{lemma}

\begin{proof}
For  notational brevity, we shall write $(\bar Y_{n\delta})_{n\ge0} $  in lieu of $(\bar Y_{n\delta}^{{\bf0}})_{n\ge0}$ as long as   the initial value $\bar Y_0={\bf0}$ is not emphasized.
According to the definition of $\delta_4$ given in \eqref{EE16}, we obviously have for any $\delta\in(0,\delta_4]$,
\begin{align}\label{EE19}
(1+L_1^2)\delta^{1-\theta}\le L_3/2\quad \mbox{ and } \quad L_3\delta /2
\le1.
\end{align}

Below, we shall fix $\delta\in(0,\delta_4]$ so that
\eqref{EE19} is valid and intend to show that
\eqref{EE18} holds true for  any integer $p\ge6,$
which, by H\"older's inequality, will imply the desired assertion for any integer $p>0$.
Via L\'{e}vy's characterization for Brownian motions, we infer that for any $t\ge0,$
\begin{align}\label{EE20}
\hat W_t^\vv:=\int_0^t\Pi( \bar Z_s )h_\vv(|\bar Z_s|)\d \bar W_s+\int_0^t\big(1-h_\vv(|\bar Z_s|)^2\big)^{{1}/{2}}\,\d \bar B_s
\end{align}
is a $d$-dimensional Brownian motion. Thus, the discrete-time version of $(\bar Y_t)_{t\ge0}$ can be written as follows: for any integer $n\ge0, $
\begin{align*}
\bar Y_{(n+1)\delta}=\bar Y_{n\delta}+b^{(\delta)}(\bar Y_{n\delta})\delta+\triangle \hat W_{n\delta}^\vv.
\end{align*}
It is easy to see that
\begin{align}\label{EE20}
\big|\bar Y_{(n+1)\delta}\big|^2=&\big|\bar Y_{n\delta}\big|^2+\Lambda (\bar Y_{n\delta})\delta+\Psi(n,\delta),
\end{align}
where
\begin{align*}
\Lambda (x):&=\big|b^{(\delta)}(x)\big|^2\delta+2\<x,b^{(\delta)}(x)\>,\quad x\in\R^d,\\
\Psi(n,\delta):&=\big|\triangle \hat W_{n\delta}^\vv\big|^2+2\<\bar Y_{n\delta}+b^{(\delta)}(\bar Y_{n\delta})\delta,\triangle \hat W_{n\delta}^\vv\>.
\end{align*}
According to the definition of $b^{(\delta)}$ given in \eqref{E7}, it follows that for all $x\in\R^d,$
\begin{align*}
\Lambda (x)=\frac{1}{1+\delta^{\theta}\varphi(|x|)}\bigg(2 \<x,b (x)\>+\frac{|b (x) |^2\delta}{1+\delta^{\theta}\varphi(|x|)}\bigg).
\end{align*}
Next, from \eqref{EE9} and \eqref{EE11}, there exists a     constant  $C_1 >0 $ such that
\begin{align*}
|b(x)|^2\le (L_1^2+1)\big(1+\varphi(|x|)\big)^2|x|^2+C_1,\quad x\in\R^d,
\end{align*}
and subsequently  there exist constants $C_2,C_3>0$ such that
\begin{align*}
2\<x,b(x)\>+\frac{|b(x)|^2\delta}{1+\delta^{\theta}\varphi(|x|)}
&\le   C_2-L_3(1+\varphi(|x|))|x|^2  +\frac{|b(x)|^2\delta}{1+\delta^{\theta}\varphi(|x|)}\\
&\le C_3+\bigg(- L_3 +\frac{ (L_1^2+1)  (1+\varphi(|x|) )\delta}{1+\delta^{\theta}\varphi(|x|)}\bigg)(1+\varphi(|x|))|x|^2\\
&\le C_3+\big(- L_3 +  (L_1^2+1)  \delta^{1-\theta}\big)(1+\varphi(|x|))|x|^2\\
&\le C_3-\frac{1}{2}L_3(1+\varphi(|x|))|x|^2,\quad x\in\R^d,
\end{align*}
where
the last inequality holds true from \eqref{EE19}. Thus, we find that
\begin{align*}
\Lambda(x)&\le \frac{1}{1+\delta^{\theta}\varphi(|x|)}\Big(C_3-\frac{1}{2}L_3(1+\varphi(|x|))|x|^2\Big) \le C_3-\frac{1}{2}L_3|x|^2,\quad x\in\R^d,
\end{align*}
where in the second inequality we used the fact that, for $\delta\in(0,1], r\mapsto (1+r)/(1+\delta^\theta r)$ is increasing on the interval $[0,\8)$. The above estimate, along with \eqref{EE20},  enables us to derive that
\begin{align*}
\big|\bar Y_{(n+1)\delta}\big|^2\le (1-L_3\delta/2)\big|\bar Y_{n\delta}\big|^2+C_3\delta+\Psi(n,\delta).
\end{align*}
The estimate above, along with the binomial theorem,
 further implies that for any integer $q\ge3,$
\begin{align}\label{moment_bound}
\big|\bar Y_{(n+1)\delta}\big|^{2q}&\le (1-L_3\delta/2)^q\big|\bar Y_{n\delta}\big|^{2q}+q(1-L_3\delta/2)^{q-1}\big|\bar Y_{n\delta}\big|^{2(q-1)}\big(C_3\delta+\Psi(n,\delta)\big)\\
&\quad+\sum_{k=0}^{q-2}C_q^k(1-L_3\delta/2)^k\big|\bar Y_{n\delta}\big|^{2k}\big(C_3\delta+\Psi(n,\delta)\big)^{q-k},
\end{align}
where the quantity  $\<\bar Y_{n\delta}, \triangle \hat W_{n\delta}^\vv\>$ involved in $\Psi(n,\delta)$ is a lower order term (providing the order $\delta^{1/2}$).
Since $\bar Y_{n\delta}$ is independent of $\triangle \hat W_{n\delta}^\vv$,
we have
\begin{align*}
\E\big(\big|\bar Y_{n\delta}\big|^{2(q-1)}\Psi(n,\delta)\big)=\E\big(\big|\bar Y_{n\delta}\big|^{2(q-1)}\big|\triangle \hat W_{n\delta}^\vv\big|^2 \big)=d\delta\E\big|\bar Y_{n\delta}\big|^{2(q-1)}.
\end{align*}
Next, \eqref{EE9} yields that there are constants $C_4,C_5>0 $ such that
\begin{align}\label{EE21}
|b^{(\delta)}(x)|\delta\le C_4\delta^{1-\theta}|x|+C_5\delta\le C_4|x|+C_5,\quad x\in\R^d.
\end{align}
In the  last term on the right hand side of \eqref{moment_bound},
the degree of the polynomial with respect to $|\bar Y_{n\delta}\big|$ is $2q-4$ at most (which obviously is a lower order
term compared with the leading term $|\bar Y_{n\delta}\big|^{2q}$), and moreover the moment of the polynomial with respect to $|\triangle \hat W_{n\delta}^\vv|$ (where the underlying degree is $2$ at least) provides at least the order $\delta$. Hence, we can deduce
from these facts and the Young inequality
that for some constant $C_6>0,$
\begin{align*}
\E\big|\bar Y_{(n+1)\delta}\big|^{2q}&\le (1-L_3\delta/4) \E\big|\bar Y_{n\delta}\big|^{2q}+ C_6\delta.
\end{align*}
Whereafter, an inductive argument, along with $\bar Y_0={\bf0}$,  implies that for any integer $n\ge0,$
\begin{align*}
\E\big|\bar Y_{(n+1)\delta}\big|^{2q}\le 4C_6/L_3.
\end{align*}
Consequently, \eqref{EE18} is attainable for any $p\ge6.$
\end{proof}

Finally, as we explained at the beginning of this section,
with the aid of Proposition \ref{pro} and Lemma \ref{lem5.2}, we carry out the proof of Theorem \ref{IPM}.
\begin{proof}[Proof of Theorem $\ref{IPM}$]
Note that,  due to \eqref{EE14}, we can verify the assumptions of \cite[Corollary 2]{Eberle} and conclude
 that there exist constants $C^*_4,\lambda_4>0$ such that for all $t\ge0$ and $\mu,\nu\in\mathscr P_1(\R^d)$,
\begin{align}\label{EE6}
\mathbb W_1\big(\mathscr L_{X_t^\mu},\mathscr L_{X_t^\nu}\big)\le C^*_4\e^{-\lambda_4 t}\mathbb W_1(\mu,\nu),
\end{align}
where $(X_t^\mu)_{t\ge0}$ stands for the solution to \eqref{E1} with the initial distribution $\mathscr L_{X_0}=\mu.$
Subsequently, via the Banach fixed point theorem,
the Markov process $(X_t)_{t\ge0}$
admits a unique invariant probability measure
$\pi\in\mathscr P_1(\R^d) $ with a finite first moment.
Likewise, by means of \eqref{W}, $(X_{n\delta}^\delta)_{n\ge0}$, which is a time-homogeneous Markov chain,  also admits a unique invariant probability measure $\pi^\delta$
 with a finite first moment.

Recall that due to \eqref{main_estimate_section6},
\begin{equation*}
\mathbb W_1\big(\pi,\pi^{\delta}\big) \le \varphi_1(n,\delta)+\varphi_2(n,\delta)+\varphi_3(n,\delta).
\end{equation*}
Below, we shall write $(X_{n\delta}^{x})_{n\ge0}$ and $(X_{n\delta}^{\delta,x})_{n\ge0}$
instead of  $(X_{n\delta}^{\mu})_{n\ge0}$ and  $(X_{n\delta}^{\delta,\mu})_{n\ge0}$ in case of $\mu=\delta_x$ for $x\in\R^d.$
By invoking \eqref{EE6} and Theorem \ref{W1}, we have
\begin{equation*}
\begin{split}
\varphi_1(n,\delta)+\varphi_3(n,\delta)&\le \int_{\R^d\times\R^d}\mathbb W_1\Big(\mathscr L_{X_{n\delta}^{x}},\mathscr L_{X_{n\delta}^{y}}\Big)\pi(\d x)\delta_{{\bf0}}(\d y)\\
&\quad+\int_{\R^d\times\R^d}\mathbb W_1\Big(\mathscr L_{X_{n\delta}^{\delta,x}},\mathscr L_{X_{n\delta}^{\delta,y}}\Big)\pi^\delta(\d x)\delta_{{\bf0}}(\d y)\\
&\le  C^*_4\e^{-\lambda_4 n\delta}\int_{\R^d\times\R^d}|x-y|\pi(\d x)\delta_{{\bf0}}(\d y)+C_3^*\e^{-\lambda_3 n\delta}\int_{\R^d\times\R^d}|x-y|\pi^\delta(\d x)\delta_{{\bf0}}(\d y)\\
&=
C^*_4\e^{-\lambda_4 n\delta}\int_{\R^d}|x|\pi(\d x) +C_3^*\e^{-\lambda_3 n\delta}\int_{\R^d}|x|\pi^{\delta}(\d x),
\end{split}
\end{equation*}
where $\int_{\R^d}|x|\pi(\d x)$ and $\int_{\R^d}|x|\pi^{\delta}(\d x)$ are both finite, since $\pi$ and $\pi^{\delta}$ have finite first moments.
This thus yields $\varphi_1(n,\delta)+\varphi_3(n,\delta)\to0$ as $n\to\8.$ Next, taking \eqref{EE15} into consideration enables us to derive that for any integer $n\ge0,$
 \begin{align}\label{EE23}
 \varphi_2(n,\delta)\le  C_*\int_0^{n\delta}\e^{-\lambda_*(n\delta-s)} \E\big|b(\bar Y_s^{{\bf0}})-b^{(\delta)}(\bar Y_{\lfloor s/\delta\rfloor\delta}^{{\bf0}})\big|\,\d s.
 \end{align}

According to the definition of $b^{(\delta)}$ given in \eqref{E7}, we infer from \eqref{EE9} and
 \eqref{ED}
that there exist constants $c_{1},c_2>0$ such that for any $x,y\in\R^d,$
\begin{align*}
\big|b(x)-b^{(\delta)}(y)\big|&\le \big|b(x)-b(y)\big|+\big|b(y)- b^{(\delta)}(y)\big|\\
&\le L_1\big(1+\varphi(|x|)+\varphi(|y|)\big)|x-y|+c_1\delta^\theta  \big(1+|y|+|y|\varphi(|y|)\big)\varphi(|y|)
\\
&\le c_2\big(1+|x|^{l_*}+|y|^{l_*}\big)|x-y|+c_2\delta^\theta\big(1+|y|^{1+2l^{*}}\big).
\end{align*}
Therefore,  H\"older's inequality
implies  that
there exists a constant $c_3>0$ such that  for any $t\ge0,$
 \begin{equation}\label{EE22}
 \begin{split}
&\E\big|b\big(\bar Y_t^{{\bf0}}\big)-b^{(\delta)}\big(\bar Y_{\lfloor t/\delta\rfloor\delta}^{{\bf0}}\big)\big|\\
&\le
c_3\Big(1+\big(\E\big|\bar Y_t^{{\bf0}}-\bar Y_{\lfloor t/\delta\rfloor\delta}^{{\bf0}}\big|^{2l_*}\big)^{{1}/{2}}+
\big(\E\big|\bar Y_{\lfloor t/\delta\rfloor\delta}^{{\bf0}}\big|^{2l_*}\big)^{{1}/{2}} \Big)
\big(\E\big|\bar Y_t^{{\bf0}}-\bar Y_{\lfloor t/\delta\rfloor\delta}^{{\bf0}}\big|^2\big)^{{1}/{2}} \\
&\quad+c_3\delta^\theta\big(1+\E\big|\bar Y_{\lfloor t/\delta\rfloor\delta}^{{\bf0}}\big|^{1+2l^{*}}\big).
\end{split}
\end{equation}
Furthermore, note that
\begin{align*}
\bar Y_t^{{\bf0}}-\bar Y_{\lfloor t/\delta\rfloor\delta}^{{\bf0}}=b^{(\delta)}\big(\bar Y_{\lfloor t/\delta\rfloor\delta}^{{\bf0}}\big)\big(t-\lfloor t/\delta\rfloor\delta\big)+\hat W_t^\vv-\hat W_{\lfloor t/\delta\rfloor\delta}^\vv,
\end{align*}
where $(\hat W_t^\vv)_{t\ge0}$ was defined in \eqref{EE20}. Thus, by virtue of \eqref{EE18} and the first inequality \eqref{EE21}, it follows  that for some constants $c_4,c_5>0$ and any $t\ge0,$
\begin{align*}
\E\big|\bar Y_t^{{\bf0}}-\bar Y_{\lfloor t/\delta\rfloor\delta}^{{\bf0}}\big|^{2l_*}\le c_4\quad \mbox{ and } \quad
\E\big|\bar Y_t^{{\bf0}}-\bar Y_{\lfloor t/\delta\rfloor\delta}^{{\bf0}}\big|^2\le c_5\delta
.
\end{align*}
As a consequence, by plugging the estimates above into \eqref{EE22}, we deduce from \eqref{EE23} there exists a constant $c_6>0$ such that
$ \varphi_2(n,\delta)\le c_6\delta^\theta$ for any integer $n\ge0$.

On the basis of the preceding analysis, we have demonstrated
the desired assertion \eqref{EE13}.
\end{proof}

\  \

\noindent \textbf{Acknowledgements.}
The research of Jianhai Bao is supported
by the National Key R\&D Program of China (2022YFA1006000) and the National Natural Science Foundation of China (No.\ 12071340).
The research of Jian Wang is supported by the National Key R\&D Program of China (2022YFA1006003) and  the National Natural Science Foundation of China (Nos.\ 12071076 and 12225104).

\end{document}